\documentclass[final,onefignum,onetabnum]{siamart190516}

\usepackage{lipsum}
\usepackage{bm}
\usepackage{amsfonts}
\usepackage{graphicx}
\usepackage{epstopdf}
\usepackage{algorithmic}
\ifpdf
  \DeclareGraphicsExtensions{.eps,.pdf,.png,.jpg}
\else
  \DeclareGraphicsExtensions{.eps}
\fi

\newsiamremark{remark}{Remark}
\newsiamremark{hypothesis}{Hypothesis}
\crefname{hypothesis}{Hypothesis}{Hypotheses}
\newsiamthm{claim}{Claim}

\headers{Lasso hyperinterpolation over general regions}{C. An and H.-N. Wu}

\title{Lasso hyperinterpolation over general regions\thanks{Submitted to the editors DATE.
}}

\author{Congpei An\thanks{School of Economic Mathematics, Southwestern University of Finance and Economics, Chengdu, China 
  (\email{ancp@swufe.edu.cn}).}
\and Hao-Ning Wu\thanks{Department of Mathematics, The University of Hong Kong, Hong Kong, China
  (\email{hnwu@hku.hk}).}}

\usepackage{amsopn}

\ifpdf
\hypersetup{
  pdftitle={Lasso hyperinterpolation},
  pdfauthor={C. An and H.-N. Wu}
}
\fi

\begin{document}

\maketitle

\begin{abstract}
This paper develops a fully discrete soft thresholding polynomial approximation over a general region, named Lasso hyperinterpolation. This approximation is an $\ell_1$-regularized discrete least squares approximation under the same conditions of hyperinterpolation. Lasso hyperinterpolation also uses a high-order quadrature rule to approximate the Fourier coefficients of a given continuous function with respect to some orthonormal basis, and then it obtains its coefficients by acting a soft threshold operator on all approximated Fourier coefficients. Lasso hyperinterpolation is not a discrete orthogonal projection, but it is an efficient tool to deal with noisy data. We theoretically analyze Lasso hyperinterpolation for continuous and smooth functions. The principal results are twofold: the norm of the Lasso hyperinterpolation operator is bounded independently of the polynomial degree, which is inherited from hyperinterpolation; and the $L_2$ error bound of Lasso hyperinterpolation is less than that of hyperinterpolation when the level of noise becomes large, which improves the robustness of hyperinterpolation. Explicit constructions and corresponding numerical examples of Lasso hyperinterpolation over intervals, discs, spheres, and cubes are given.
\end{abstract}




\section{Introduction}
\label{sec:introduction}
Hyperinterpolation over compact subsets or manifolds was introduced by Sloan in 1995 \cite{sloan1995polynomial}. Coefficients of an $L_2$ orthogonal projection from the space of continuous functions onto the space of polynomials of degree at most $L$ are expressed in the form of Fourier integrals, and hyperinterpolation of degree $L$ is constructed by approximating these integrals via a quadrature rule that exactly integrates all polynomials of degree at most $2L$. Thus hyperinterpolation is a numerical discretization of the $L_2$ orthogonal projection, and it is highly related to some spectral methods in solving differential and integral equations, which are known as discrete Galerkin methods \cite{atkinson2019spectral,graham2002fully,hansen2009norm}. In the past decades, hyperinterpolation has attracted many interests, and a lot of important works have been done, for example, see \cite{an2012regularized,caliari2007hyperinterpolation,caliari2008hyperinterpolation,dai2006generalized,de2009new,hansen2009norm,le2001uniform,pieper2009vector,reimer2000hyperinterpolation,sloan2000constructive,sloan2012filtered} and references therein.

Hyperinterpolation is a discrete least squares approximation scheme with a high-order quadrature rule assumed, which was revealed in \cite{sloan1995polynomial}, thus it requires the concerned function to be sampled on a well chosen finite set to achieve the high algebraic accuracy of the quadrature rule. With elements in such a set and corresponding sampling values of the function deemed as input and output data, respectively, studies on hyperinterpolation assert that it is an effective approach to modeling mappings from input data to output data. However, in real-world applications, one possibly has noisy samples. In this paper, we propose a novel strategy, \emph{Lasso hyperinterpolation}, with Lasso incorporated into hyperinterpolation, to handle noise. Lasso, the acronym of ``least absolute shrinkage and selection operator'', is a shrinkage and selection method for linear regression \cite{tibshirani1996regression}, which is blessed with the abilities of denoising and feature selection. Lasso hyperinterpolation is a constructive approximation: based on hyperinterpolation, Lasso hyperinterpolation proceeds all hyperinterpolation coefficients by a soft threshold operator. Thus it is not only feasible to study approximation properties of it, but also easy to implement this novel scheme.

When the level of noise is relatively small, least squares approximation is shown to be able to reduce noise \cite{hesse2017radial,MR4118851,le2008localized}. However, this method is not suitable when the level of noise becomes large. There have been attempts to improve the robustness of hyperinterpolation with respect to noise, for example, filtered hyperinterpolation \cite{sloan2012filtered} and Tikhonov regularized discrete least squares approximation \cite{an2012regularized,an2020tikhonov}. Filtered hyperinterpolation filters hyperinterpolation coefficients by some filters (for different kinds of filters, we refer to \cite{an2012regularized,filbir2008polynomial,sloan2011polynomial}), shrinking these coefficients continuously as the order of the basis element increasing. The mechanism of Tikhonov regularized least squares approximation is similar; actually, it was revealed in \cite{an2012regularized} that Tikhonov regularized least squares approximation reduces to filtered hyperinterpolation on the unit two-sphere with a certain filter. Both attempts improve the performance of hyperinterpolation in dealing with noisy samples. However, continuous filtering or shrinking may not work as well as ``Lasso hyperinterpolation'' in denoising, which proceeds these coefficients by a soft threshold operator. Apart from denoising, Lasso is also blessed with the feature selection ability. In the context of hyperinterpolation, a feature is a basis element, and feature selection is called basis element selection in this paper. Hyperinterpolation and its filtered variant do not hold such an ability, whereas Lasso hyperinterpolation can select basis elements with higher relevancy to the concerned function and dismiss the rest in order to simplify the expansion. The level of relevancy can be determined by controlling parameters in Lasso hyperinterpolation.

We will study approximation properties of Lasso hyperinterpolation and provide error analysis. Inherited from hyperinterpolation, the norm of the Lasso hyperinterpolation operator is bounded independently of the polynomial degree $L$. However, Lasso hyperinterpolation does not possess the convergence property of hyperinterpolation as $L\rightarrow\infty$. It is shown that in the absence of noise, the $L_2$ error of Lasso hyperinterpolation for any nonzero continuous function converges to a nonzero term, which depends on the best approximation of the function, whereas such an error of both hyperinterpolation and filtered hyperinterpolation converges to zero. However, in the presence of noise, Lasso hyperinterpolation is able to reduce the newly introduced error term caused by noise, via multiplying a factor less than one. Similar results are also obtained when the function is blessed with additional smoothness.

The rest of this paper is organized as follows. In Section \ref{sec:review}, we review some basic ideas of quadrature and hyperinterpolation. In Section \ref{sec:lassoandsparsity}, we display how Lasso is incorporated into hyperinterpolaion, and analyze the basis element selection ability of Lasso hyperinterpolation. In Section \ref{sec:general}, we study Lasso hyperinterpolation over general regions, presenting some properties of the Lasso hyperinterpolation operator and deriving error bounds. In this analysis, we consider two cases: Lasso hyperinterpolation of continuous functions and smooth functions, respectively. Section \ref{sec:example} focuses on four concrete examples on the interval, the unit disc, the unit two-sphere, and the unit cube, respectively, and provides some numerical examples. 

\section{Backgrounds}
\label{sec:review}
In this section, we review some basic ideas of hyperinterpolation. Let $\Omega$ be a compact and smooth Riemannian manifold in $\mathbb{R}^s$ with smooth or empty boundary and measure $\text{d}\omega$. The manifold $\Omega$ is assumed to have finite measure with respect to a given positive measure $\text{d}\omega$, that is,
\begin{equation*}
\int_{\Omega}\text{d}\omega=V<\infty.
\end{equation*}
We wish to approximate a nonzero $f\in \mathcal{C}(\Omega)$ (possibly noisy) by a polynomial in $\mathbb{P}_L$, which is the linear space of polynomials on $\Omega$ of degree at most $L$. 
\subsection{Quadratures}
With respect to the given finite measure $\text{d}\omega$, an inner product between functions $v$ and $z$ on $\Omega$ is defined as
\begin{equation}\label{equ:innerproduct}
\left<v,z\right>=\int_{\Omega}vz\text{d}\omega,
\end{equation}
which is a definition involving integrals. Quadrature, in a computational perspective, is a standard term for numerical computation of integrals, and is also one of the techniques that approximation theory can be linked to applications immediately \cite{trefethen2013approximation}. Assume that we are given a quadrature rule of the form
\begin{equation}\label{equ:quadrature}
\sum\limits_{j=1}^Nw_jg(\mathbf{x}_j)\approx\int_{\Omega}g\text{d}\omega
\end{equation}
with the property that it exactly integrates all polynomials of degree at most $2L$, where $\mathcal{X}_N:=\{\mathbf{x}_1,\ldots,\mathbf{x}_N\}\subset\Omega$ is a set of $N$ distinct points in $\Omega$ and quadrature weights $w_j$ are positive for all $1\leq j\leq N$. That is, we require
\begin{equation}\label{equ:exactquadrature}
\sum\limits_{j=1}^Nw_jg(\mathbf{x}_j)=\int_{\Omega}g\text{d}\omega\quad\forall g\in\mathbb{P}_{2L}.
\end{equation}
Based on the assumed quadrature, Sloan introduced a ``discrete inner product'' \cite{sloan1995polynomial} 
\begin{equation}\label{equ:discreteinnerproduct}
\left<v,z\right>_{N}:=\sum_{j=1}^Nw_jv(\mathbf{x}_j)z(\mathbf{x}_j),
\end{equation}
corresponding to the ``continuous'' inner product \eqref{equ:innerproduct}.

\subsection{Hyperinterpolation}
Hyperinterpolation is a discretization of the $L_2$ orthogonal projection $\mathcal{P}_Lf$ of $f\in\mathcal{C}(\Omega)$ onto $\mathbb{P}_L$. Let $d:=\dim\mathbb{P}_L$ be the dimension of $\mathbb{P}_L$, and let $\{p_1,\ldots,p_{d}\}\subset\mathbb{P}_L$ be an orthonormal basis of $\mathbb{P}_L$, that is,
\begin{equation}\label{equ:delta1}
\left<p_i,p_j\right>=\delta_{ij},\quad1\leq i,j\leq d,
\end{equation}
where $\delta_{ij}$ is the Kronecker delta. Thus $\mathcal{P}_Lf$ is defined as
\begin{equation}
\mathcal{P}_Lf:=\sum_{\ell=1}^{d}\left<f,p_{\ell}\right>p_{\ell}.
\end{equation}
If the integral is evaluated by the quadrature rule, then hyperinterpolation $\mathcal{L}_Lf$ is defined analogously to $\mathcal{P}_Lf$.
\begin{definition}[\cite{sloan1995polynomial}]
Given a quadrature rule \eqref{equ:quadrature} with exactness \eqref{equ:exactquadrature}. A \emph{hyperinterpolation} of $f$ onto $\mathbb{P}_L$ is defined as
\begin{equation}\label{equ:hyperinterpolation}
\mathcal{L}_Lf:=\sum_{\ell=1}^{d}\left<f,p_{\ell}\right>_Np_{\ell}.
\end{equation}
\end{definition}

As the most degree $2L$ of $p_ip_j$ ensures the exactness \eqref{equ:exactquadrature} of the quadrature, we have $\left<p_i,p_j\right>_N=\left<p_i,p_j\right>=\delta_{ij}$, hence the discrete inner product \eqref{equ:discreteinnerproduct} also satisfies
\begin{equation}\label{equ:delta2}
\left<p_i,p_j\right>_N=\delta_{ij},\quad1\leq i,j\leq d.
\end{equation}
It follows the least number of quadrature points such that the quadrature \eqref{equ:quadrature} is exact:
\begin{lemma}[Lemma 2 in \cite{sloan1995polynomial}]\label{lem:leastnumber}
If a quadrature rule is exact for all polynomials of degree at most $2L$, the number of quadrature points $N$ should satisfy $N\geq \dim\mathbb{P}_L=d$.
\end{lemma}
\begin{definition}\label{def:minimal}
An $N$-point quadrature rule which is exact for all polynomials of degree at most $2L$ is called \emph{minimal} if $N=d$.
\end{definition}
There are two important and practical properties of $\mathcal{L}_L$, one is that it has the classical interpolation property if and only if the quadrature is minimal, the other is that it becomes exact if $f$ is a polynomial in $\mathbb{P}_L$. 
\begin{lemma}[Lemma 3 in \cite{sloan1995polynomial}]\label{lem:lemma3}
The classical interpolation formula 
\begin{equation}
\mathcal{L}_Lf(\mathbf{x}_j)=f(\mathbf{x}_j),\quad 1\leq j\leq N,
\end{equation}
holds for arbitrary $f\in\mathcal{C}(\Omega)$ if and only if the quadrature rule is minimal.
\end{lemma}
\begin{lemma}[Lemma 4 in \cite{sloan1995polynomial}]\label{lem:lemma4}
If $f\in\mathbb{P}_L$ then $\mathcal{L}_Lf=f$.
\end{lemma}

\subsection{Filtered hyperinterpolation}\label{sec:filtered}
Filtered hyperinterpolation \cite{sloan2011polynomial,sloan2000constructive,sloan2012filtered,wang2017filtered}, roughly speaking, makes use of a ``filter'' function $h\in\mathcal{C}(\mathbb{R}^+)$: all coefficients $\left<f,p_{\ell}\right>_N$ are filtered by $h$, that is, they are multiplied by $h({\deg p_{\ell}}/L)$, where $\deg p_{\ell}$ is the degree of $p_{\ell}$. A filter function $h$ satisfies
\begin{equation*}
h(x)=\begin{cases}
1 &\text{ for }x\in[0,1/2],\\
0&\text{ for }x\in[1,\infty),
\end{cases}
\end{equation*}
and $h$ on $[1/2,1]$ has various definitions (see, e.g. \cite{an2012regularized,sloan2011polynomial}) with the continuity of $h$ been ensured. A trigonometric polynomial filter \cite{an2012regularized} is used for numerical study in this paper, which defines $h(x)=\sin^2\pi x$ on $[1/2,1]$.
\begin{definition}
Given a quadrature rule \eqref{equ:quadrature} with exactness \eqref{equ:exactquadrature}. A \emph{filtered hyperinterpolation} of $f$ onto $\mathbb{P}_{L-1}$ is defined as
\begin{equation}\label{equ:filteredhyperinterpolation}
\mathcal{F}_Lf=\sum_{\ell=1}^{d}h\left(\frac{\deg p_{\ell}}{L}\right)\left<f,p_{\ell}\right>_Np_{\ell}.
\end{equation}
\end{definition}
From \eqref{equ:filteredhyperinterpolation}, we have $\mathcal{F}_Lf=f$ for all $f\in\mathbb{P}_{\lfloor L/2\rfloor}$, 
where $\lfloor\cdot\rfloor$ is the floor function. Recently, it was shown in \cite{lin2019distributed} that (distributed) filtered hyperinterpolation can reduce weak noise on spherical functions. In this paper, we will compare the ability of denoising between filtered hyperinterpolation and the following Lasso hyperinterpolation, see examples in Section \ref{sec:example}.

\section{Lasso hyperinterpolation}\label{sec:lassoandsparsity}
Lasso mainly aims at denoising and feature selection, and it has always been investigated in a discrete way in statistics, optimization, compressed sensing, and so forth \cite{tibshirani1996regression,tibshirani2011regression}. Thus it is natural, and feasible as well, to introduce Lasso into hyperinterpolation to handle noisy data, and to simplify the hyperinterpolation polynomial by dismissing basis elements of less relevance to concerned function $f$. 

\subsection{Formulation}
To introduce Lasso hyperinterpolation, we first reveal that $\mathcal{L}_Lf$ is a solution to a least squares approximation problem, which was first stated by Sloan in 1995 \cite{sloan1995polynomial}. For the sake of completeness, we give a proof of this remarkable result, which is stated in Lemma \ref{prop:hyperinterpolation}. Consider the following discrete least squares approximation problem 
\begin{equation}\label{equ:approximationproblem}
\min\limits_{p\in\mathbb{P}_L}~~\left\{\frac12\sum_{j=1}^Nw_j\left(p(\mathbf{x}_j)-f(\mathbf{x}_j)\right)^2\right\}\quad\text{with}\quad p(\mathbf{x})=\sum_{\ell=1}^{d}\alpha_{\ell}p_{\ell}(\mathbf{x})\in\mathbb{P}_L.
\end{equation}
Let $\mathbf{A}\in\mathbb{R}^{sN\times d}$ be a matrix with elements $[\mathbf{A}]_{j\ell}=p_{\ell}(\mathbf{x}_j)$, $j=1,\ldots,N$ and $\ell=1,\ldots,d$ (recall $\mathbf{x}_j\in\mathbb{R}^s$), and let $\mathbf{W}$ be a diagonal matrix with entries $\{w_{j}\}_{j=1}^N$. The approximation problem \eqref{equ:approximationproblem} can be transformed into an equivalent approximation problem
\begin{equation}\label{equ:optimizationproblem}
\min\limits_{\bm{\alpha}\in\mathbb{R}^{d}}~~\frac12\|\mathbf{W}^{1/2}(\mathbf{A}\bm{\alpha}-\mathbf{f})\|_2^2,
\end{equation}
where $\bm{\alpha}=[\alpha_1,\ldots,\alpha_d]^{\text{T}}\in\mathbb{R}^d$ is a collection of coefficients $\{\alpha_{\ell}\}_{\ell=1}^{d}$ in constructing $p$, and $\mathbf{f}=[f(\mathbf{x}_1),\ldots,f(\mathbf{x}_N)]^{\text{T}}\in\mathbb{R}^N$ is a vector of sampling values $\{f(\mathbf{x}_j)\}_{j=1}^N$ on $\mathcal{X}_N$. Since problem \eqref{equ:optimizationproblem} is a strictly convex problem, the stationary point of the objective is none other than the unique solution to \eqref{equ:optimizationproblem}. Taking the first derivative of the objective in problem \eqref{equ:optimizationproblem} with respect to $\bm{\alpha}$ leads to the first-order condition
\begin{equation}\label{equ:firstorder}
\mathbf{A}^{\text{T}}\mathbf{WA}\bm{\alpha}-\mathbf{A}^{\text{T}}\mathbf{Wf}=\mathbf{0}.
\end{equation}
Note that the assumption $f\neq0$ implies $\|\mathbf{A}^{\rm{T}}\mathbf{Wf}\|_{\infty}\neq0$. With the first-order condition \eqref{equ:firstorder}, we have the following result. 
\begin{lemma}[Lemma 5 in \cite{sloan1995polynomial}]\label{prop:hyperinterpolation}
Given $f\in\mathcal{C}(\Omega)$, let $\mathcal{L}_Lf\in\mathbb{P}_L$ be defined by \eqref{equ:hyperinterpolation}, where the quadrature points (all in $\Omega$) and weights (all positive) in the discrete inner product satisfies the exactness property \eqref{equ:exactquadrature}. Then $\mathcal{L}_Lf$ is the unique solution to the approximation problem \eqref{equ:approximationproblem}.
\end{lemma}
\begin{proof}
The proof is based on the first-order condition \eqref{equ:firstorder}. On the one hand, the matrix $\mathbf{A}^{\text{T}}\mathbf{WA}$ is an identity matrix as all entries of it satisfy 
\begin{equation*}
[\mathbf{A}^{\text{T}}\mathbf{WA}]_{ik}=\sum_{j=1}^Nw_jp_i(\mathbf{x}_j)p_k(\mathbf{x}_j)=\left<p_i,p_k\right>_N=\delta_{ik},\quad 1\leq i,k\leq d,
\end{equation*}
where the last equality holds due to property \eqref{equ:delta2}. On the other hand, the vector $\mathbf{A}^{\text{T}}\mathbf{Wf}$ is in fact a collection of discrete inner products:
\begin{equation*}
[\mathbf{A}^{\text{T}}\mathbf{Wf}]_{\ell}=\sum_{j=1}^Nw_jp_{\ell}(\mathbf{x}_j)f(\mathbf{x}_j)=\left<p_{\ell},f\right>_N,\quad \ell=1,\ldots,d.
\end{equation*}
Hence the polynomial constructed with coefficients $\alpha_{\ell}=\left<p_{\ell},f\right>_N$ is indeed $\mathcal{L}_Lf$. The uniqueness is due to the strict convexity of problem \eqref{equ:optimizationproblem}.
\end{proof}

Now we start to involve Lasso into $\mathcal{L}_L$. From the original idea of Lasso \cite{tibshirani1996regression}, it restricts the sum of absolute values of coefficients to be bounded by some positive number, say $\eta$. Then for $p=\sum_{\ell=1}^{d}\alpha_{\ell}p_{\ell}$, incorporating Lasso into $\mathcal{L}_L$ can be achieved via solving the constrained least squares problem
\begin{equation}\label{equ:lassoapproximationconstrained}
\min\limits_{p\in\mathbb{P}_L}~~\left\{\frac12\sum_{j=1}^Nw_j\left(p(\mathbf{x}_j)-f(\mathbf{x}_j)\right)^2\right\}\quad\text{subject to}\quad\sum_{\ell=1}^{d}|\alpha_{\ell}|\leq\eta.
\end{equation}
Solving this problem is equivalent to solving the following regularized least squares approximation problem
\begin{equation}\label{equ:lassoapproximation}
\min\limits_{p\in\mathbb{P}_L}~~\left\{\frac12\sum_{j=1}^Nw_j\left(p(\mathbf{x}_j)-f(\mathbf{x}_j)\right)^2+\lambda\sum_{\ell=1}^{d}|\mu_{\ell}\beta_{\ell}|\right\}\quad\text{with}\quad p=\sum_{\ell=1}^{d}\beta_{\ell}p_{\ell}\in\mathbb{P}_L,
\end{equation}
where $\lambda>0$ is the \emph{regularization parameter} and $\{\mu_{\ell}\}_{\ell=1}^{d}$ is a set of positive \emph{penalty parameters}. We make two comments on problem \eqref{equ:lassoapproximation}: we adopt new notation $\beta_{\ell}$ instead of using $\alpha_{\ell}$ in order to distinguish the coefficients of Lasso hyperinterpolation from those of $\mathcal{L}_L$; and we introduce a sequence $\{\mu_{\ell}\}_{\ell=1}^{d}$ of penalty parameters into the model so that the model could be more general and more flexible. The solution to problem \eqref{equ:lassoapproximationconstrained} is also a solution to problem \eqref{equ:lassoapproximation} when $\mu_{\ell}=1$ for all $\ell=1,\ldots,d$.

The solution to problem \eqref{equ:lassoapproximation} is our \emph{Lasso hyperinterpolation polynomial}. We directly give the definition of Lasso hyperinterpolation first, and then show that the Lasso hyperinterpolation polynomial is indeed the solution to \eqref{equ:lassoapproximation}. To describe it, we need the notion of \emph{soft threshold operator}.
\begin{definition}\label{def:sto}
The \emph{soft threshold operator}, denoted by $\mathcal{S}_{k}(a)$, is defined as $\mathcal{S}_k(a):=\max(0,a-k)+\min(0,a+k)$.
\end{definition}
We add $\lambda$ as a superscript into $\mathcal{L}_Lf$, denoting that this is a regularized version (Lasso regularized) of $\mathcal{L}_L$ with regularization parameter $\lambda$.
\begin{definition}\label{def:lassohyperinterpolation}
Given a quadrature rule \eqref{equ:quadrature} with exactness \eqref{equ:exactquadrature}. A \emph{Lasso hyperinterpolation} of $f$ onto $\mathbb{P}_L$ is defined as
\begin{equation}\label{equ:lassohyperinterpolation}
\mathcal{L}_L^{\lambda}f:=\sum_{\ell=1}^{d}\mathcal{S}_{\lambda\mu_{\ell}}\left(\left<f,p_{\ell}\right>_N\right)p_{\ell},\quad \lambda>0.
\end{equation}
\end{definition}
The logic of Lasso hyperinterpolation is to process each coefficient $\left<f,p_{\ell}\right>_N$ of hyperinterpolation by a soft threshold operator $\mathcal{S}_{\lambda\mu_{\ell}}(\cdot)$. Then we revisit problem \eqref{equ:lassoapproximation}. 

Let $\mathbf{R}\in\mathbb{R}^{d\times d}$ be a diagonal matrix with entries $\{\mu_{\ell}\}_{\ell=1}^{d}$. Similar to the process from \eqref{equ:approximationproblem} to \eqref{equ:optimizationproblem}, problem \eqref{equ:lassoapproximation} can also be transformed into 
\begin{equation}\label{equ:lassooptimization}
\min\limits_{\bm{\beta}\in\mathbb{R}^{d}}~~\frac12\|\mathbf{W}^{1/2}(\mathbf{A}\bm{\beta}-\mathbf{f})\|_2^2+\lambda\|\mathbf{R}\bm{\beta}\|_1,\quad \lambda>0,
\end{equation}
where $\bm{\beta}=[\beta_1,\ldots,\beta_{d}]^{\text{T}}\in\mathbb{R}^{d}$. As the convex term $\|\cdot\|_1$ is nonsmooth, some subdifferential calculus of convex functions \cite{MR2511061} is needed. Then we have the following result.
\begin{theorem}\label{thm:lassohyperinterpolation}
Let $\mathcal{L}_L^{\lambda}f\in\mathbb{P}_L$ be defined by \eqref{equ:lassohyperinterpolation} and adopt conditions of Lemma \ref{prop:hyperinterpolation}. Then $\mathcal{L}_L^{\lambda}f$ is the solution to the regularized least squares approximation problem \eqref{equ:lassoapproximation}.
\end{theorem}
\begin{proof}
In Lemma \ref{prop:hyperinterpolation} we have proved that $\mathbf{A}^{\text{T}}\mathbf{WA}$ is an identity matrix, and by \eqref{equ:firstorder} we have $\bm{\alpha}=\mathbf{A}^{\text{T}}\mathbf{Wf}$. Then $\bm{\beta}=[\beta_1,\ldots,\beta_{d}]^{\text{T}}$ is a solution to \eqref{equ:lassooptimization} if and only if
\begin{equation}\label{equ:lassofirstorder}
\mathbf{0}\in\mathbf{A}^{\text{T}}\mathbf{WA}\bm{\beta}-\mathbf{A}^{\text{T}}\mathbf{Wf}+\partial\left(\|\mathbf{R}\bm{\beta}\|_1\right)=
\bm{\beta}-\bm{\alpha}+\partial\left(\|\mathbf{R}\bm{\beta}\|_1\right),
\end{equation}
where $\partial(\cdot)$ denotes the subdifferential. The first order condition \eqref{equ:lassofirstorder} is equivalent to
\begin{equation}\label{equ:1storder}
0\in\beta_{\ell}-\alpha_{\ell}+\lambda\mu_{\ell}\partial(|\beta_{\ell}|)\quad\forall\ell=1,\ldots,{d},
\end{equation}
where
\begin{equation*}
\partial(|\beta_{\ell}|)=
\begin{cases}
1 & \text{ if }\beta_{\ell}>0,\\
-1 & \text{ if }\beta_{\ell}<0,\\
\in[-1,1] & \text{ if }\beta_{\ell}=0.
\end{cases}
\end{equation*}
If we denote by $\bm{\beta}^{*}=[\beta_1^*,\ldots,\beta_{d}^*]^{\text{T}}$ the optimal solution to \eqref{equ:lassooptimization}, then
\begin{equation*}
\beta_{\ell}^*=\alpha_{\ell}-\lambda\mu_{\ell}\partial(|\beta_{\ell}^*|),\quad\ell=1,\ldots,d.
\end{equation*}
Thus there are three cases should be considered:
\begin{enumerate}
  \item[1)] If $\alpha_{\ell}>\lambda\mu_{\ell}$, then $\alpha_{\ell}-\lambda\mu_{\ell}\partial(|\beta_{\ell}^*|)>0$, thus $\beta_{\ell}^*>0$, yielding $\partial(|\beta_{\ell}^*|)=1$, then $\beta_{\ell}^*=(\alpha_{\ell}-\lambda\mu_{\ell})>0$.
  \item[2)] If $\alpha_{\ell}<-\lambda\mu_{\ell}$, then $\alpha_{\ell}+\lambda\mu_{\ell}\partial(|\beta_{\ell}^*|)<0$, which leads to $\beta_{\ell}^*<0$, giving $\partial(|\beta_{\ell}^*|)=-1$, then $\beta_{\ell}^*=(\alpha_{\ell}+\lambda\mu_{\ell})<0$.
  \item[3)] Consider now $-\lambda\mu_{\ell}\leq\alpha_{\ell}\leq\lambda\mu_{\ell}$. On the one hand, $\beta_{\ell}^*>0$ leads to $\partial(|\beta_{\ell}^*|)=1$, and then $\beta_{\ell}\leq0$; on the other hand, $\beta_{\ell}^*<0$ leads to $\partial(|\beta_{\ell}^*|)=-1$, and then $\beta_{\ell}\geq0$. Two contradictions enforce $\beta_{\ell}^*$ to be $0$.
\end{enumerate}
Recall that $\alpha_{\ell}=\left<f,p_{\ell}\right>_{N}$ for all $\ell=1,\ldots,d$. With all cases considered, the polynomial constructed with coefficients $\beta_{\ell}=\mathcal{S}_{\lambda\mu_{\ell}}(\alpha_{\ell})$, $\ell=1,\ldots,d$ is indeed $\mathcal{L}_L^{\lambda}f$. 
\end{proof}

There are three important facts of Lasso hyperinterpolation, which distinguish it from hyperinterpolation. 
\begin{remark}
Note that $\mathcal{S}_{\lambda\mu_{\ell}}(\alpha_{\ell})\neq \alpha_{\ell}$. Thus even though the quadrature rule is minimal, Lasso hyperinterpolation does not satisfy the classical interpolation property, i.e., $\mathcal{L}_L^{\lambda}f(\mathbf{x}_j)\neq f(\mathbf{x}_j),1\leq j\leq N$. But by Lemma \ref{lem:lemma3}, $\mathcal{L}_Lf(\mathbf{x}_j)=f(\mathbf{x}_j),1\leq j\leq N$.
\end{remark}
\begin{remark}
The fact $\mathcal{S}_{\lambda\mu_{\ell}}(\alpha_{\ell})\neq \alpha_{\ell}$ also implies that $\mathcal{L}^{\lambda}_Lp\neq p$ for all $p\in\mathbb{P}_L$. Hence $\mathcal{L}_L^{\lambda}$ is not a projection operator, as $\mathcal{L}^{\lambda}_L(\mathcal{L}^{\lambda}_Lf)\neq\mathcal{L}^{\lambda}_Lf$ for any nonzero $f\in\mathcal{C}(\Omega)$. However, by Lemma \ref{lem:lemma4} we have $\mathcal{L}_Lp= p$ and hence $\mathcal{L}_L(\mathcal{L}_Lf)=\mathcal{L}_Lf$ for all $f\in\mathcal{C}(\Omega)$.
\end{remark}
\begin{remark}
Lasso hyperinterpolation is not invariant under a change of basis. That is, suppose $\{q_1,q_2,\ldots,q_d\}$ is another basis, due to the existence of the soft threshold operator, $\mathcal{L}_L^{\lambda}f$ cannot be expressed by $\sum_{i,j=1}^d\mathcal{S}_{\lambda\mu_{j}}\left(\left<f,q_j\right>_N\right)\left[\mathbf{Q}^{-1}\right]_{ji}q_i$, where $\mathbf{Q}\in\mathbb{R}^{d\times d}$ is a matrix with elements $[\mathbf{Q}]_{ij}=\left<q_i,q_j\right>_N$. However, we have $\mathcal{L}_Lf=\sum_{i,j=1}^d\left<f,q_j\right>_N\left[\mathbf{Q}^{-1}\right]_{ji}q_i$ \cite{sloan1995polynomial}.
\end{remark}

Lasso hyperinterpolation is also different from filtered hyperinterpolation.
\begin{remark}
Lasso hyperinterpolation processes hyperinterpolation coefficients via a soft threshold operator, which processes them in a discontinuous way; however, filtered hyperinterpolation processes them in a continuous way.
\end{remark}
\subsection{Basis element selection and parameter choice}
Basis element selection ability of $\mathcal{L}_L^{\lambda}$ stems from the soft threshold operator, which enforces any $\left<f,p_{\ell}\right>_{N}$ to be $0$ so long as $\left<f,p_{\ell}\right>_{N}\leq \lambda\mu_{\ell}$, and shrinks the rest by subtracting $\lambda\mu_{\ell}$ from them. The following proposition states which basis elements would be dismissed.
\begin{proposition}\label{prop:sparsity}
Under conditions of Theorem \ref{thm:lassohyperinterpolation}, given $\lambda$ and $\{\mu_{\ell}\}_{\ell=1}^{d}$, coefficients corresponding to $p_{\ell}$ which satisfies $|\left<f,p_{\ell}\right>_{N}|\leq \lambda\mu_{\ell}$ are enforced to be $0$. 
\end{proposition}
\begin{proof}
Given in the discussion above.
\end{proof}
\noindent That is, basis elements corresponding to these coefficients would be dismissed, and the rest are kept in constructing a hyperinterpolation polynomial. 

To quantify the basis element selection ability of $\mathcal{L}_L^{\lambda}$, we investigate the sparsity of $\bm{\beta}$, measured by the ``zero norm'' $\|\bm{\beta}\|_0$ which is the number of nonzero entries of $\bm{\beta}$.
\begin{theorem}\label{thm:sparsity}
Under conditions of Theorem \ref{thm:lassohyperinterpolation}, let $\bm{\beta}$ be a solution to problem \eqref{equ:lassooptimization}.
\begin{enumerate}
  \item[1)] If $\lambda=0$, then $\|\bm{\beta}\|_0$ satisfies $\|\bm{\beta}\|_0=\|\bm{\alpha}\|_0=\|\mathbf{A}^{\rm{T}}\mathbf{Wf}\|_0$.
  \item[2)] If $\lambda>0$, then $\|\bm{\beta}\|_0$ satisfies $\|\bm{\beta}\|_0\leq\|\mathbf{A}^{\rm{T}}\mathbf{Wf}\|_0$, more precisely,
      \begin{equation*}
       \|\bm{\beta}\|_0=\|\mathbf{A}^{\rm{T}}\mathbf{Wf}\|_0-\#\left\{\ell:|\left<f,p_{\ell}\right>_{N}|\leq\lambda\mu_{\ell}\text{ and }\left<f,p_{\ell}\right>_{N}\neq0\right\},
      \end{equation*}
      where $\#\{\cdot\}$ denotes the cardinal number of the corresponding set.
\end{enumerate}
\end{theorem}
\begin{proof}
If $\lambda=0$, then $\mathcal{L}_L^{\lambda}f$ reduces to $\mathcal{L}_{L}f$, hence $\bm{\beta}=\bm{\alpha}=\mathbf{A}^{\rm{T}}\mathbf{Wf}$, proving assertion 1). Given nonzero entries of $\mathbf{A}^{\rm{T}}\mathbf{Wf}$, as stated in Proposition \ref{prop:sparsity}, Lasso hyperinterpolation enforces those $\left<f,p_{\ell}\right>_{N}$ satisfying $|\left<f,p_{\ell}\right>_{N}|\leq\lambda\mu_{\ell}$ to be zero. Hence assertion 2) holds obviously.
\end{proof}

\begin{remark}
If we measure the level of relevancy between certain basis element and the function $f$ by the absolute value of $\left<f,p_{\ell}\right>_{N}$, then Theorem \ref{thm:sparsity} suggests that we can determine a baseline of such a level of relevancy and dismiss those basis elements with lower relevancy by controlling parameters $\lambda$ and $\{\mu_{\ell}\}_{\ell=1}^{d}$. 
\end{remark}

In an extreme case, we could even set large enough $\lambda$ and $\{\mu_{\ell}\}_{\ell=1}^{d}$ so that all coefficients $\left<f,p_{\ell}\right>_N$ are enforced to be 0. However, as we are constructing a polynomial to approximate some $f\in\mathcal{C}(\Omega)$, this is definitely not the case we desire. Then we have the rather simple but interesting result, a parameter choice rule for $\lambda$ such that $\bm{\beta}\neq \mathbf{0}$.
\begin{theorem}\label{prop:rule}
Adopt conditions of Theorem \ref{thm:lassohyperinterpolation}. If $\lambda<\|\mathbf{A}^{\rm{T}}\mathbf{Wf}\|_{\infty}$, then $\bm{\beta}$ obtained by solving \eqref{equ:lassooptimization} is not $\mathbf{0}\in\mathbb{R}^d$. 
\end{theorem}
\begin{proof}
Suppose to contrary that $\mathbf{0}\in\mathbb{R}^d$ is a stationary point of the objective in \eqref{equ:lassooptimization} , then the first-order condition \eqref{equ:lassofirstorder} with $\mathbf{0}$ gives $\mathbf{A}^{\rm{T}}\mathbf{Wf}\in\lambda\partial(\|\mathbf{0}\|_1)=\lambda[-1,1]^{d}$, leading to $\|\mathbf{A}^{\rm{T}}\mathbf{Wf}\|_{\infty}\leq\lambda$. Hence its contrapositive also holds: if $\lambda<\|\mathbf{A}^{\rm{T}}\mathbf{Wf}\|_{\infty}$, then $\bm{\beta}$ could not be $\mathbf{0}$.
\end{proof}

\section{Error analysis}\label{sec:general}
In this section, the theory of Lasso hyperinterpolation is developed. The denoising ability of $\mathcal{L}_L^{\lambda}$ is measured by the $L_2$ error bounds. One of our two main results is that Lasso can reduce the operator norm of $\mathcal{L}_L$. The other main result is that Lasso hyperinterpolation can reduce the error related to noise. We consider additive noise in this paper, that is, Lasso hyperinterpolation finds an approximation polynomial to $f\in\mathcal{C}(\Omega)$ with noisy data values $f^{\epsilon}(\mathbf{x}_j)=f(\mathbf{x}_j)+\epsilon_j$ at points $\mathbf{x}_j\in \Omega$. It is convenient to regard $f^{\epsilon}$ as a continuous function on $\Omega$, which can be constructed by some interpolation process from values $\{f^{\epsilon}(\mathbf{x}_j)\}_{j=1}^N$ on $\mathcal{X}_N$. 

We first derive $L_2$ error bounds for continuous functions $f\in\mathcal{C}(\Omega)$ with noise-free and noisy data values, respectively. Then we make further discussion on the type of noise. Finally we consider the relation between additional smoothness of $f$ and obtained error bounds. Norms of functions and operators used in our analysis are defined below. For any function $g\in\mathcal{C}(\Omega)$, its uniform norm is defined as $\|g\|_{\infty}:=\sup_{\mathbf{x}\in\Omega}|g(\mathbf{x})|$, and for any $g\in L_2(\Omega)$, its $L_2$ norm is defined as
$\|g\|_2:=(\int_{\Omega}|g|^2\text{d}\omega)^{1/2}$. For any operator $\mathcal{U}_L:\mathcal{C}{(\Omega)}\rightarrow L_2(\Omega)$, its operator norm is defined as 
\begin{equation*}
\|\mathcal{U}_L\|_{\text{op}}:=\sup_{g\in\mathcal{C}(\Omega),g\neq0}\frac{\|\mathcal{U}_Lg\|_{2}}{\|g\|_{\infty}}.
\end{equation*}
For $g\in\mathcal{C}(\Omega)$, let $\varphi^*\in\mathbb{P}_L$ be the best approximation of $g$ in $\mathbb{P}_L$, that is, $E_L(g):=\inf_{\varphi\in\mathbb{P}_L}\|g-\varphi\|_{\infty}=\|g-\varphi^*\|_{\infty}$.

\subsection{The case of continuous functions}
Recall that $\int_{\Omega}\text{d}\omega=V<\infty$. We first state error bounds for $\mathcal{L}_L$ for comparison. 
\begin{proposition}[Theorem 1 in \cite{sloan2012filtered}]
Suppose conditions of Lemma \ref{prop:hyperinterpolation} are assumed. Then
\begin{equation}\label{equ:stability}
\|\mathcal{L}_Lf\|_2\leq V^{1/2}\|f\|_{\infty},
\end{equation}
and 
\begin{equation}\label{equ:error}
\|\mathcal{L}_Lf-f\|_2\leq2V^{1/2}E_L(f).
\end{equation}
Thus $\|\mathcal{L}_Lf-f\|_2\rightarrow0$ as $L\rightarrow\infty$.
\end{proposition}

We have the following lemma to describe properties of $\mathcal{L}_L^{\lambda}$.
\begin{lemma}\label{lem:mainlemma}
Under conditions of Theorem \ref{thm:lassohyperinterpolation},
\begin{itemize}
 \item[(1)]$\left<f-\mathcal{L}_{L}^{\lambda}f,\mathcal{L}_{L}^{\lambda}f\right>_{N}=K(f)$,
  \item[(2)] $\left<\mathcal{L}_{L}^{\lambda}f,\mathcal{L}_{L}^{\lambda}f\right>_{N}+\left<f-\mathcal{L}_{L}^{\lambda}f,f-\mathcal{L}_{L}^{\lambda}f\right>_{N}=
\left<f,f\right>_{N}-2K(f)$, and thus $K(f)$ satisfies $K(f)\leq\left<f,f\right>_{N}/2$,
  \item[(3)]$\left<\mathcal{L}_{L}^{\lambda}f,\mathcal{L}_{L}^{\lambda}f\right>_{N}\leq\left<f,f\right>_{N}-2K(f)$,
\end{itemize}
where
\begin{equation*}
K(f)=\sum_{\ell=1}^{d}\left(\mathcal{S}_{\lambda\mu_{\ell}}(\alpha_{\ell})\alpha_{\ell}-\left(\mathcal{S}_{\lambda\mu_{\ell}}(\alpha_{\ell})\right)^2\right)\geq0
\end{equation*}
and $K(f)=0$ if $\lambda=0$ or if $\lambda$ is so large that $|\alpha_{\ell}|\leq\lambda\mu_{\ell}$ for all $\ell$. Here $K(f)$ is a constant relying on $f$, by noting that $\alpha_{\ell}=\sum_{j=1}^Nw_jp_{\ell}(\mathbf{x}_j)f(\mathbf{x}_j)$.
\end{lemma}
\begin{proof}
The positiveness of $K(f)$ stems from $|\alpha_{\ell}|\geq|\mathcal{S}_{\lambda\mu_{\ell}}(\alpha_{\ell})|$ and from the fact that they have the same signs if $\mathcal{S}_{\lambda\mu_{\ell}}(\alpha_{\ell})\neq0$. When $\lambda=0$ we have $\mathcal{S}_{\lambda\mu_{\ell}}(\alpha_{\ell})=\alpha_{\ell}$, and when $\lambda$ is so large that $|\alpha_{\ell}|\leq\lambda\mu_{\ell}$ we have $\mathcal{S}_{\lambda\mu_{\ell}}(\alpha_{\ell})=0$, both make $K(f)$ be 0.\\
(1) This follows from $\left<f-\mathcal{L}_{L}^{\lambda}f,\mathcal{L}_{L}^{\lambda}f\right>_{N}=\sum\limits_{k=1}^{d}\mathcal{S}_{\lambda\mu_{k}}(\alpha_{k})\left<f-\sum\limits_{\ell=1}^{d}\mathcal{S}_{\lambda\mu_{\ell}}(\alpha_{\ell})p_{\ell},p_k\right>_N$ and
\begin{equation*}
\left<f-\sum_{\ell=1}^{d}\mathcal{S}_{\lambda\mu_{\ell}}(\alpha_{\ell})p_{\ell},p_k\right>_N=\left<f,p_k\right>_N-\left<\sum_{\ell=1}^{d}\mathcal{S}_{\lambda\mu_{\ell}}(\alpha_{\ell})p_{\ell},p_k\right>_N=\alpha_{k}-\mathcal{S}_{\lambda\mu_{k}}(\alpha_{k}).
\end{equation*}
(2) It follows from (1) that $\left<\mathcal{L}_{L}^{\lambda}f,\mathcal{L}_{L}^{\lambda}f\right>_N=\left<f,\mathcal{L}_{L}^{\lambda}f\right>_{N}-K(f)$, and the second term of the left-hand side can be written as 
\begin{equation*}
\left<f-\mathcal{L}_{L}^{\lambda}f,f-\mathcal{L}_{L}^{\lambda}f\right>_{N}=\left<f,f\right>_N-2\left<f,\mathcal{L}_{L}^{\lambda}f\right>_N+\left<\mathcal{L}_{L}^{\lambda}f,\mathcal{L}_{L}^{\lambda}f\right>_N. 
\end{equation*}
Summing them up and using (1) again lead to the equality. It follows from $\left<g,g\right>_N\geq0$ for any $g\in\mathcal{C}(\Omega)$ that $\left<f,f\right>_{N}-2K(f)\geq0$, thus we obtain the upper bound of $K(f)$.\\
(3) This is immediately from (2) and the positiveness of $\left<f-\mathcal{L}_{L}^{\lambda}f,f-\mathcal{L}_{L}^{\lambda}f\right>_{N}$.
\end{proof}

\begin{remark}
$K(f)=0$ implies either $\lambda=0$ or $\lambda$ is so large that $\mathcal{L}_L^{\lambda}f=0$. However, Definition \ref{def:lassohyperinterpolation} of Lasso hyperinterpolation requires $\lambda>0$, and the parameter choice rule for $\lambda$ described in Theorem \ref{prop:rule} prevents $\mathcal{L}_L^{\lambda}f$ to be $0$. Thus $K(f)>0$ always holds if $\lambda$ is chosen appropriately with respect to Theorem \ref{prop:rule}.
\end{remark}

In a noise-free case, we show that $\mathcal{L}_L^{\lambda}$ can reduce the stability estimation \eqref{equ:stability} and reduce the factor 2 in the error estimation \eqref{equ:error} of $\mathcal{L}_L$, but it introduces an additional term into the error bound, which we call it an \emph{regularization error}. The term $K(f)$ in Lemma \ref{lem:mainlemma} will be used in our estimation.
\begin{theorem}\label{prop:noisefree}
Adopt conditions of Theorem \ref{prop:rule}. Then there exists $\tau_1<1$, which relies on $f$ and is inversely related to $K(f)$, such that
\begin{equation}\label{equ:lassostability}
\|\mathcal{L}_L^{\lambda}f\|_2\leq \tau_1 V^{1/2}\|f\|_{\infty},
\end{equation}
where $V=\int_{\Omega}\rm{d}\omega$, and there exists $\tau_2<1$, which relies on $f$ and is inversely related to $K(f-\varphi^*)$, such that
\begin{equation}\label{equ:lassoerror}
\|\mathcal{L}^{\lambda}_Lf-f\|_2\leq(1+\tau_2)V^{1/2}E_{L}(f)+\|\varphi^*-\mathcal{L}_L^{\lambda}\varphi^*\|_{2},
\end{equation}
where $\varphi^*$ is the best approximation of $f$ in $\mathbb{P}_L$ over $\Omega$. 
\end{theorem}
\begin{proof}
Inequality \eqref{equ:lassostability} follows from
\begin{equation*}\begin{split}
\left\|\mathcal{L}_L^{\lambda}f\right\|_2^2&=\left<\mathcal{L}_L^{\lambda}f,\mathcal{L}_L^{\lambda}f\right>
=\left<\mathcal{L}_L^{\lambda}f,\mathcal{L}_L^{\lambda}f\right>_{N}
\leq\left<f,f\right>_{N}-2K(f)\\
&=\sum_{j=1}^Nw_jf(\mathbf{x}_j)^2-2K(f)\leq\sum_{j=1}^Nw_j\|f\|_{\infty}^2-2K(f)=V\|f\|_{\infty}^2-2K(f),
\end{split}\end{equation*}
where in the second equation we use the fact $\mathcal{L}_L^{\lambda}f\in\mathbb{P}_L$, and the next inequality is due to Lemma \ref{lem:mainlemma}(3). As $K(f)>0$, there exists $\tau_1=\tau_1(K(f))<1$, which is inversely related to $K(f)$, such that 
\begin{equation*}
\sqrt{V\|f\|_{\infty}^2-2K(f)}=\tau_1 V^{1/2}\|f\|_{\infty}. 
\end{equation*}
Then for any polynomial $\varphi\in\mathbb{P}_L$, 
\begin{equation*}\begin{split}
\|\mathcal{L}_L^{\lambda}f-f\|_{2}&=\|\mathcal{L}_L^{\lambda}(f-\varphi)-(f-\varphi)-(\varphi-\mathcal{L}_L^{\lambda}\varphi)\|_{2}\\
&\leq\|\mathcal{L}_L^{\lambda}(f-\varphi)\|_{2}+\|f-\varphi\|_{2}+\|\varphi-\mathcal{L}_L^{\lambda}\varphi\|_{2}.
\end{split}\end{equation*}
Since the inequality holds for arbitrary $\varphi\in\mathbb{P}_L$, we let $\varphi=\varphi^*$. Then there exists $\tau_2=\tau_2(K(f-\varphi^*))<1$, which is inversely related to $K(f-\varphi^*)$, such that
\begin{equation*}
\|\mathcal{L}_L^{\lambda}f-f\|_{2}\leq\tau_2 V^{1/2}\|f-\varphi^*\|_{\infty}+V^{1/2}\|f-\varphi^*\|_{\infty}+\|\varphi^*-\mathcal{L}_L^{\lambda}\varphi^*\|_{2},
\end{equation*}
where the second term on the right side is due to Cauchy-Schwarz inequality, which ensures $\|g\|_{2}=\sqrt{\left<g,g\right>}\leq\|g\|_{\infty}\sqrt{\left<1,1\right>}=V^{1/2}\|g\|_{\infty}$ for all $g\in\mathcal{C}(\Omega)$. Hence we obtain the error bound \eqref{equ:lassoerror}.
\end{proof}

Inequality \eqref{equ:lassostability} gives $\|\mathcal{L}_L^{\lambda}\|_{\text{op}}\leq\tau_1 V^{1/2}$, showing the norm of $\mathcal{L}_L^{\lambda}$ is less than that of hyperinterpolation $\mathcal{L}_L$. For the error estimation \eqref{equ:lassoerror}, passing to the limit of $L$ gives the following limit case
\begin{equation*}
\lim_{L\rightarrow\infty}\|\mathcal{L}^{\lambda}_Lf-f\|_2\leq\lim_{L\rightarrow\infty}\|\varphi^*-\mathcal{L}_L^{\lambda}\varphi^*\|_{2}\neq0,
\end{equation*}
due to the fact that $\mathcal{L}^{\lambda}_{L}\varphi\neq \varphi$ for all $\varphi\in\mathbb{P}_L$. Only when $\lambda\rightarrow0$, we can have $\lim_{L\rightarrow\infty}\|\varphi^*-\mathcal{L}_L^{\lambda}\varphi^*\|_{2}=0$ because $\mathcal{L}_L\varphi^*=\varphi^*$, suggested by Lemma \ref{lem:lemma4}.

\begin{remark}
Comparing with the stability result \eqref{equ:stability} and the error bound \eqref{equ:error} of $\mathcal{L}_L$, it is shown that Lasso hyperinterpolation can reduce both of them, but an additional regularization error $\|\varphi^*-\mathcal{L}_L^{\lambda}\varphi^*\|_{2}$ is introduced in a natural manner. In general, we do not recommend the use of Lasso in the absence of noise. However, if the data values are noisy, then $\mathcal{L}_L^{\lambda}$ will play an important part in reducing noise.
\end{remark}

The following theorem describes the denoising ability of $\mathcal{L}_L^{\lambda}$.
\begin{theorem}\label{thm:noisybound}
Adopt conditions of Theorem \ref{prop:rule}. Assume $f^{\epsilon}\in\mathcal{C}(\Omega)$ is a noisy version of $f$, and let $\mathcal{L}_L^{\lambda}f^{\epsilon}\in\mathbb{P}_L$ be defined by \eqref{equ:lassohyperinterpolation}. Then there exists $\tau_3<1$, which relies on $f$ and $f^{\epsilon}$, and is inversely related to $K(f^{\epsilon}-\varphi^*)$, such that
\begin{equation}\label{equ:lassoerrornoise}
\|\mathcal{L}_L^{\lambda}f^{\epsilon}-f\|_2\leq\tau_3 V^{1/2}\|f-f^{\epsilon}\|_{\infty}+(1+\tau_3)V^{1/2}E_{L}(f)+\|\varphi^*-\mathcal{L}_L^{\lambda}\varphi^*\|_{2},
\end{equation}
where $V=\int_{\Omega}\rm{d}\omega$ and $\varphi^*$ is the best approximation of $f$ in $\mathbb{P}_L$ over $\Omega$. 
\end{theorem}
\begin{proof}
For any polynomial $\varphi\in\mathbb{P}_L$, we have
\begin{equation*}\begin{split}
\|\mathcal{L}_L^{\lambda}f^{\epsilon}-f\|_{2}&=\|\mathcal{L}_L^{\lambda}(f^{\epsilon}-\varphi)-(f-\varphi)-(\varphi-\mathcal{L}_L^{\lambda}\varphi)\|_{2}\\
&\leq\|\mathcal{L}_L^{\lambda}(f^{\epsilon}-\varphi)\|_{2}+\|f-\varphi\|_{2}+\|\varphi-\mathcal{L}_L^{\lambda}\varphi\|_{2}.
\end{split}\end{equation*}
Then by Theorem \ref{prop:noisefree}, letting $\varphi=\varphi^*$ gives 
\begin{equation*}
\|\mathcal{L}_L^{\lambda}f^{\epsilon}-f\|_{2}\leq\tau_3 V^{1/2}\|f^{\epsilon}-\varphi^*\|_{\infty}+V^{1/2}\|f-\varphi^*\|_{\infty}+\|\varphi^*-\mathcal{L}_L^{\lambda}\varphi^*\|_{2},
\end{equation*}
where $\tau_3<1$ is inversely related to $K(f^{\epsilon}-\varphi^*)$. Estimating $\|f^{\epsilon}-\varphi^*\|_{\infty}$ by $\|f^{\epsilon}-\varphi^*\|_{\infty}\leq\|f^{\epsilon}-f\|_{\infty}+\|f-\varphi^*\|_{\infty}$ gives \eqref{equ:lassoerrornoise}.
\end{proof}
\begin{remark}
If Lasso is not incorporated, using the stability result \eqref{equ:stability} of $\mathcal{L}_L$ gives the following estimation which describes how $\mathcal{L}_L$ handles noisy functions:
\begin{equation}\label{equ:hypernoise}
\|\mathcal{L}_Lf^{\epsilon}-f\|_2\leq V^{1/2}\|f-f^{\epsilon}\|_{\infty}+2V^{1/2}E_{L}(f),
\end{equation}
which enlarges the part $\tau_3 V^{1/2}\|f-f^{\epsilon}\|_{\infty}+(1+\tau_3)V^{1/2}E_{L}(f)$ in \eqref{equ:lassoerrornoise} but vanishes the regularization error. In principle, there should be a trade-off choice strategy for $\lambda$; and in practice, when the level of noise is of a significant scale, denoising is at the top priority, and the regularization error now has little to contribute to the total error bound. 
\end{remark}

\subsection{A discussion on noise}\label{sec:noise}
We have obtained error bounds of $\mathcal{L}_L^{\lambda}$ when $f\in\mathcal{C}(\Omega)$, and we now continue to discuss the term $\|f-f^{\epsilon}\|_{\infty}$ with respect to different kinds of noise. 

Let $\bm{\epsilon}=[\epsilon_1,\epsilon_2,\ldots,\epsilon_N]^{\text{T}}\in\mathbb{R}^N$ be a vector of noise added onto $\{f(\mathbf{x}_j)\}_{j=1}^N$, where $\|\bm{\epsilon}\|_{\infty}=\max_j|\epsilon_j|$. It is natural to assume that $\|f-f^{\epsilon}\|_{\infty}=\|\bm{\epsilon}\|_{\infty}$, which means that we adopt the deterministic noise model and allow the worst noise level to be at any point of $\mathcal{X}_N$. This assumption was suggested in \cite{pereverzev2015parameter}, which simplifies the estimation of $\|f-f^{\epsilon}\|_{\infty}$ and provides a possible way to study different types of noise in sampling $f^{\epsilon}(\mathbf{x}_j)$. With the randomness of $\bm{\epsilon}$, we can establish error estimations of $\mathcal{L}_L^{\lambda}$ in the form of mathematical expectations, in which $\mathbb{E}(\|f-f^{\epsilon}\|_{\infty})$ is transformed into $\mathbb{E}(\|\bm{\epsilon}\|_{\infty})=\mathbb{E}(\max_{j}|\epsilon_j|)$. If the distribution of noise is known, then the term $\mathbb{E}\left(\max_j|\epsilon_j|\right)$ can be estimated analytically. In this paper, for example, if we let $\epsilon_j$ be a \emph{sub-Gaussian random variable} \cite[Section 2.5]{MR3837109}, which is equipped with a \emph{sub-Gaussian norm} $$\|\epsilon_j\|_{\psi_2}:=\inf\{T>0:\mathbb{E}(\exp(\epsilon_j^2/T^2)\leq2)\},$$
then $\mathbb{E}(\|\bm{\epsilon}\|_{\infty})$ can be estimated by the following lemma.
\begin{lemma}\cite[Section 2.5.2]{MR3837109}\label{lem:noiseestimation}
If $\epsilon_1,\ldots,\epsilon_N$ are a sequence of sub-Gaussian random variables, which are not necessarily independent, then
\begin{equation}\label{equ:noiseestimation}
\mathbb{E}(\|\bm{\epsilon}\|_{\infty})\leq cG\sqrt{\log{N}},
\end{equation}
where $G=\max_j\|\epsilon_j\|_{\psi_2}$ and $c>0$ is a generic constant. 
\end{lemma}

The family of sub-Gaussian random variables contains many classical examples \cite{MR3837109}. For instance, if $\epsilon_j\in\mathcal{N}(0,\sigma_j^2)$ is a Gaussian random variable with mean zero and variance $\sigma_j^2$, we have $\|\epsilon_j\|_{\psi^2}\leq c\sigma_j$, and hence $$\mathbb{E}(\|\bm{\epsilon}\|_{\infty})\leq c\max_{j=1,\ldots,N}\sigma_j\sqrt{\log{N}},$$ which describes a wide range of Gaussian noise. If $\epsilon_j$ is a single impulse, that is, $\epsilon_j=a_j$ with probability $b_j$ and $\epsilon_j=0$ with probability $1-b_j$, then $\|\epsilon_j\|_{\psi^2}=a_j/\sqrt{\ln(2/b_j)}$, and hence $$\mathbb{E}(\|\bm{\epsilon}\|_{\infty})\leq c\max_{j=1,\ldots,N}\frac{a_j}{\sqrt{\ln(2/b_j)}}\sqrt{\log{N}}.$$
One can also obtain similar bounds if $\epsilon_j$ is a multiple impulse by calculating the sub-Gaussian norm of $\epsilon_j$. These bounds cover the case of impulse noise (also known as salt-and-pepper noise in imaging science). More generally, as long as $\epsilon_j$ is a bounded random variable with $|\epsilon_j|\leq d_j$, we can obtain $\|\epsilon_j\|_{\psi^2}\leq d_j/\sqrt{\ln{2}}$, and hence $$\mathbb{E}(\|\bm{\epsilon}\|_{\infty})\leq c\max_{j=1,\ldots,N}\frac{d_j}{\sqrt{\ln{2}}}\sqrt{\log{N}}.$$
Moreover, we can note that the estimation \eqref{equ:noiseestimation} is also valid for mixed noise, so long as $\epsilon_1,\ldots,\epsilon_N$ are all sub-Gaussian random variables and $G=\max_j\|\epsilon_j\|_{\psi_2}$.

\subsection{The case of smooth functions}\label{sec:discussion}
We now set up particular estimations on terms $E_L(f)$ and $\|\varphi^*-\mathcal{L}_L^{\lambda}\varphi^*\|_{2}$ if $f$ is assumed to be blessed with additional smoothness. To measure the smoothness, it is convenient to introduce a H\"{o}lder space 
\begin{equation*}
\mathcal{C}^{k,\zeta}(\Omega):=\left\{g\in\mathcal{C}^k(\Omega):D^mg \text{ is }\zeta-\text{H\"{o}lder continuous } \forall m\text{ with }|m|=k\right\}
\end{equation*}
such that $f\in\mathcal{C}^{k,\zeta}(\Omega)$, where $D$ is a differential operator and $m$ ranges over multi-indices; and $f$ could also be considered in a Sobolev space $H^{k+\zeta+s/2}$, which is continuously embedded in $\mathcal{C}^{k,\zeta}$ \cite{sobolev}. Note that it is not necessary to assume any additional smoothness on $f^{\epsilon}$, which shall still belong to $\mathcal{C}(\Omega)$. Then the term $E_L(f)$ in both bounds \eqref{equ:lassoerror} and \eqref{equ:lassoerrornoise} can be quantified by $k$ with the aid of some Jackson type theorems \cite{MR2511061,ragozin1970polynomial,ragozin1971constructive}. Generally speaking, for $f\in\mathcal{C}^{k,\zeta}(\Omega)$ with $0<\zeta\leq1$, there exists $C(k,\zeta,s)>0$, which depends only on $k$, $\zeta$, and $s$, such that \cite{ragozin1970polynomial}
\begin{equation}\label{equ:best}
E_L(f)\leq C(k,\zeta,s)\|f\|_{k,\zeta}L^{-k-\zeta}=\mathcal{O}(L^{-k-\zeta}),
\end{equation}
where $$\|f\|_{k,\zeta}=\sum_{|m|\leq k}\left\|D^mf\right\|_{\infty}+\sum_{|m|=k}\sup_{\mathbf{x}\neq\mathbf{y}}\frac{|f(\mathbf{x})-f(\mathbf{y})|}{\|\mathbf{x}-\mathbf{y}\|_2^{\zeta}}$$
and $m$ ranges over multi-indices. To obtain \eqref{equ:best}, it is assumed that the $k$th derivative of $f$ satisfies a $\zeta$-H\"{o}lder condition. This modulus was also generalized in \cite{ragozin1970polynomial}, but the convergence rate $\mathcal{O}(L^{-k-\zeta})$ is not affected. In particular, if $f\in\mathcal{C}^k([-1,1])$ and if $|f^{k}(x_1)-f^{(k)}(x_2)|\leq M_k|x_1-x_2|^{\zeta}$ for some $M_k>0$ and $\zeta\in(0,1]$, Jackson theorem \cite[Theorem 3.7.2]{MR2511061} asserts $E_L(f)=\max_{-1\leq x\leq 1}|f(x)-\varphi^*(x)|\leq C(k,\zeta)M_kL^{-k-\zeta}$ for some $C(k,\zeta)>0$, which depends only on $k$ and $\zeta$ (as $s=1$). If $\Omega$ is a cube or a multi-dimensional torus, one can also find similar Jackson type estimations in \cite[Section 6.4]{MR0213785}. 

If $\Omega$ is blessed with some additional geometric properties, then the requirement $f\in\mathcal{C}^{k,\zeta}(\Omega)$ may be relaxed to $f\in\mathcal{C}^{k}(\Omega)$ and the bound for $E_{L}(f)$ becomes 
\begin{equation}\label{equ:additionalgeometric}
E_L(f)\leq\mathcal{O}(L^{-k})
\end{equation} 
correspondingly. For example, if $\Omega$ is a homogeneous submanifold (including spheres and projective spaces), then there exists a polynomial such that the $L_2$ distance from $f$ to this polynomial is bounded by $\mathcal{O}(L^{-k})$, hence the bound \eqref{equ:additionalgeometric} is valid for $E_L(f)$ \cite{ragozin1971constructive}. In particular, if $f\in\mathcal{C}^k(\mathbb{S}^{s-1})$, where $\mathbb{S}^{s-1}\subset\mathbb{R}^s$ is a unit $(s-1)$-sphere, a Jackson type theorem \cite[Theorem 3.3]{ragozin1971constructive} asserts $E_L(f)$ satisfies \eqref{equ:additionalgeometric}. Besides, though the closed unit $s-$ball $\mathbb{B}^s\subset\mathbb{R}^s$ is not a homogeneous submanifold, a Jackson type theorem can be also derived based on results on the unit sphere \cite[Theorem 3.4]{ragozin1971constructive}, which states that the bound \eqref{equ:additionalgeometric} is also valid for $f\in\mathcal{C}^k(\mathbb{B}^s)$. For detailed mathematical derivation and constants used in $\mathcal{O}(L^{-k})$, we refer to \cite{ragozin1971constructive}.
\begin{remark}
In this paper, $E_L(f)$ is defined in the sense of uniform norm. As we mentioned above, $f$ can be considered in some Sobolev spaces $H^t(\Omega)$ (or more generally $W^{t,p}(\Omega)$) continuously embedded in $\mathcal{C}^{k,\zeta}(\Omega)$. Thus in the literature of hyperinterpolation (mainly on spheres), errors and $E_L(f)$ were also studied in the $H^t(\Omega)$ sense of Sobolev norm $\|\cdot\|_{H^t(\Omega)}$ with $E_L(f):=\inf_{p\in\mathbb{P}_L}\|f-p\|_{H^t(\Omega)}$. We refer to \cite{dai2011polynomial,MR2274179} for details about this topic.
\end{remark}

Finally we examine and estimate the regularization error $\|\varphi^*-\mathcal{L}_L^{\lambda}\varphi^*\|_{2}$ with $f\in\mathcal{C}^{k,\zeta}(\Omega)$, $0<\zeta\leq1$. No matter when $f\in\mathcal{C}(\Omega)$ or $\mathcal{C}^{k,\zeta}(\Omega)$, this term will not vanish unless $\lambda\rightarrow0$ and $f=f^{\epsilon}$, which is a fact verified in many previous works, see, e.g. \cite{lin2019distributed,pereverzev2015parameter}. This term essentially depends on $f$ through the medium of its best uniform approximation polynomial $\varphi^*$. If $\varphi^*$ is constructed as $\varphi^*=\sum_{\ell=1}^dc_{\ell}p_{\ell}$, then corresponding coefficients of $\mathcal{L}^{\lambda}_L\varphi^*$ are $c_{\ell}-\lambda\mu_{\ell}$ if $c_{\ell}>\lambda\mu_{\ell}$, $c_{\ell}+\lambda\mu_{\ell}$ if $c_{\ell}<-\lambda\mu_{\ell}$, and $0$ if $|c_{\ell}|\leq\lambda\mu_{\ell}$. Thus with the aid of Parseval's identity in $\mathbb{P}_L$, we have
\begin{equation}\label{equ:boundingregularizationerror}
\|\varphi^*-\mathcal{L}^{\lambda}_L\varphi^*\|_2=\left(\sum_{\ell=1,~|c_{\ell}|\leq\lambda\mu_{\ell}}^d|c_{\ell}|^2+\sum_{\ell=1,~|c_{\ell}|>\lambda\mu_{\ell}}^d|\lambda\mu_{\ell}|^2\right)^{1/2}.
\end{equation} 
Here comes an immediate but rough bound $\|\varphi^*-\mathcal{L}^{\lambda}_L\varphi^*\|_2\leq(\sum_{\ell=1}^d|c_{\ell}|^2)^{1/2}=\|\varphi^*\|_2$, but we are going to derive a sharper bound for it in consideration of function smoothness and regularization parameters. 
\begin{lemma}\label{lem:regularizationerror}
Adopt conditions of Theorem \ref{prop:rule} and let $f\in\mathcal{C}^{k,\zeta}(\Omega)$ with $0<\zeta\leq1$. Let $\varphi^*=\sum_{\ell=1}^dc_{\ell}p_{\ell}\in\mathbb{P}_L$ be the best approximation of $f$ in the sense of uniform norm, and let $\mathcal{L}_L^{\lambda}\varphi^*\in\mathbb{P}_L$ be defined by \eqref{equ:lassohyperinterpolation}. Then 
\begin{equation*}
\|\varphi^*-\mathcal{L}_L^{\lambda}\varphi^*\|_{2}\leq \left[V\left(\|f\|_{\infty}+C\|f\|_{k,\zeta}L^{-k-\zeta}\right)^2-\chi\right]^{1/2},
\end{equation*}
where $V=\int_{\Omega}\rm{d}\omega$, $C:=C(k,\zeta,s)$ is some constant which depends only on $k$, $\zeta$ and $s$, and $$\chi:=\sum_{\ell=1,~|c_{\ell}|>\lambda\mu_{\ell}}^d\left(|c_{\ell}|^2-|\lambda\mu_{\ell}|^2\right).$$  

\end{lemma}
\begin{proof}
Comparing with the rough bound $\|\varphi^*-\mathcal{L}^{\lambda}_L\varphi^*\|_2\leq\|\varphi^*\|_2$, we exactly have 
\begin{equation*}
\|\varphi^*-\mathcal{L}^{\lambda}_L\varphi^*\|_2^2= \|\varphi^*\|_2^2-\chi.
\end{equation*}
Since $\|\varphi^*\|_2^2\leq V \|\varphi^*\|^2_{\infty} \leq V\left(\|f\|_{\infty}+C\|f\|_{k,\zeta}L^{-k-\zeta}\right)^2$, the lemma is proved.
\end{proof}
\begin{remark}
Lemma \ref{lem:regularizationerror} decomposes $\|\varphi^*-\mathcal{L}_L^{\lambda}\varphi^*\|_{2}^2$ into two parts, one depends only on $f$ itself, the other is related to regularization settings. If $f$ is smoother, characterized by a larger $k$, then $\|\varphi^*-\mathcal{L}^{\lambda}_L\varphi^*\|_2$ becomes smaller. But $\|\varphi^*-\mathcal{L}^{\lambda}_L\varphi^*\|_2$ is inversely related to $\lambda$ and $\mu_{\ell}$, which are inversely related to $\chi$.
\end{remark}

Consequently, error bounds \eqref{equ:lassoerror} and \eqref{equ:lassoerrornoise} can be improved as follows.
\begin{theorem}\label{thm:improvedbounds}
Adopt conditions of Theorem \ref{prop:rule} and Lemma \ref{lem:noiseestimation}. Let $f\in\mathcal{C}^{k,\zeta}(\Omega)$ with $0<\zeta\leq 1$ and let $f^{\epsilon}\in\mathcal{C}(\Omega)$ be a noisy version of $f$. Then 
\begin{equation*}\begin{split}
\|\mathcal{L}^{\lambda}_Lf-f\|_2\leq&\left[(1+\tau_2)V^{1/2}C\|f\|_{k,\zeta}\right]L^{-k-\zeta}\\
&+\left[V\left(\|f\|_{\infty}+C\|f\|_{k,\zeta}L^{-k-\zeta}\right)^2-\chi\right]^{1/2},
\end{split}\end{equation*}
and
\begin{equation}\label{equ:generalerrorinexpectation}\begin{split}
\mathbb{E}(\|\mathcal{L}_L^{\lambda}f^{\epsilon}-f\|_2)\leq &c\tau_3 V^{1/2}G\sqrt{\log{N}}+\left[(1+\tau_3)V^{1/2}C\|f\|_{k,\zeta}\right]L^{-k-\zeta}\\
&+\left[V\left(\|f\|_{\infty}+C\|f\|_{k,\zeta}L^{-k-\zeta}\right)^2-\chi\right]^{1/2},
\end{split}\end{equation}
where $c$ is a generic constant; $V=\int_{\Omega}\rm{d}\omega$; $C:=C(k,\zeta,s)>0$ depends on $k,\zeta$ and $s$; $\chi:=\sum_{\ell=1,~|c_{\ell}|>\lambda\mu_{\ell}}^d\left(|c_{\ell}|^2-|\lambda\mu_{\ell}|^2\right)$; $G$ could be determined analytically if the type of noise is known; $\tau_2<1$ depends on $f$ and is inversely related to $K(f-\varphi^*)$; and $\tau_3<1$ depends on $f$ and $f^{\epsilon}$ and is inversely related to $K(f^{\epsilon}-\varphi^*)$.
\end{theorem}
\begin{proof}
Based on error decompositions in \eqref{equ:lassoerror} and \eqref{equ:lassoerrornoise}, both improved error bounds can be obtained by using Lemmas \ref{lem:noiseestimation} and \ref{lem:regularizationerror}, and estimation \eqref{equ:best}.
\end{proof}
\begin{remark}
Our error bound \eqref{equ:generalerrorinexpectation} consists of three terms. The first term is related to the level of noise, which will tend to zero if $f^{\epsilon}\rightarrow f$; the second term is an essential part in almost every approximation scheme, converging to zero as $L\rightarrow \infty$; and the third term depends on our regularization settings, which cannot converge. These findings on convergence and misconvergence also apply for \eqref{equ:lassoerrornoise} when $f\in\mathcal{C}(\Omega)$.
\end{remark}

\section{Examples}\label{sec:example}
We consider four concrete examples of $\mathcal{L}_L^{\lambda}$ with certain quadrature rules: on the interval $[-1,1]\subset\mathbb{R}$ , on the unit disc $\{(x_1,x_2)\in\mathbb{R}^2:x_1^2+x_2^2\leq1\}$, on the unit sphere $\mathbb{S}^2:=\{\mathbf{x}=(x,y,z)^{\text{T}}\in\mathbb{R}^3:x^2+y^2+z^2=1\}\subset\mathbb{R}^3$, and in the unit cube $[-1,1]^3\subset\mathbb{R}^3$ as well. For each example, we state the quadrature rule and the value of $V$, thus error bounds of $\mathcal{L}_L^{\lambda}$ can be obtained immediately. We will also compare $\mathcal{L}^{\lambda}_L$ with filtered hyperinterpolation $\mathcal{F}_L$ (cf. Section \ref{sec:filtered}) and Tikhonov regularized least squares approximation $\mathcal{T}_L$, which can be obtained by using regularization term $\|\bm{\alpha}\|_2^2/2$ rather than $\|\bm{\alpha}\|_1$ in \eqref{equ:lassoapproximationconstrained}. We refer to \cite{an2012regularized,an2020tikhonov,MR4118851,pereverzev2015parameter} for this topic. Except for examples on the sphere (explanation will be made from the context), $\mathcal{T}_L$ adopts the same regularization parameters as $\mathcal{L}_L^{\lambda}$ in our numerical experiments.

\subsection{The interval}\label{sec:interval}
We take $\Omega=[-1,1]$ with $\text{d}\omega=\omega(x)\text{d}x$, where $\omega(x)\geq0$ is a weight function on $[-1,1]$ and different $\omega(x)$ indicates different value of $V=\int_{-1}^1\omega(x)\text{d}x$. In this case, $\mathbb{P}_L$ is a linear space of polynomials of degree at most $L$ on the interval $[-1,1]$, and hence $d=L+1$.

Fix $L$ as the degree of Lasso hyperinterpolation polynomial, and let $\{p_{\ell}:0\leq\ell\leq L\}$ be a family of normalized orthogonal polynomials on $[-1,1]$ with respect to a weight function $\omega(x)$, and $p_{\ell}$ is of degree $\ell$ \cite{gautschi2004orthogonal}. For $N\geq L+1$, let $\{x_j\}_{j=1}^N$ and $\{w_j\}_{j=1}^N$ be Gauss quadrature points and Gauss quadrature weights, respectively. Due to the exactness of Gauss quadrature \cite{trefethen2013approximation,xiang2012error}, it is ensured that
\begin{equation*}
\sum_{j=1}^Nw_jg(x_j)=\int_{-1}^1\omega(x)g(x)\text{d}x\quad\forall g\in\mathbb{P}_{2N-1}.
\end{equation*}
Here we take Gauss quadrature as an example, and its variants Gauss-Lobatto and Gauss-Radau quadrature may also be considered \cite{gautschi2004orthogonal}. Then the Lasso hyperinterpolation \eqref{equ:lassohyperinterpolation} becomes
\begin{equation}\label{equ:lassohyperinterpolationoninterval}
\mathcal{L}_L^{\lambda}f:=\sum_{\ell=0}^L\mathcal{S}_{\lambda\mu_{\ell}}\left(\sum_{j=1}^Nw_jf(x_j)p_{\ell}(x_j)\right)p_{\ell}.
\end{equation}

When $\lambda\rightarrow0$ and $f^{\epsilon}=f$, our error bounds of $\mathcal{L}_L^{\lambda}$ reduce into $\|\mathcal{L}_Lf-f\|_{2}\leq2V^{1/2}E_{L}(f)$, which quantifies the error of interpolation if $N=L+1$, known as Erd{\H{o}}s-Tur{\'a}n bound \cite{erdHos1937interpolation}, and is valid for hyperinterpolation if $N>L+1$, proved by Sloan \cite{sloan1995polynomial}.

Figure \ref{figure1} provides a concrete example on the approximation of function $f(x)=\exp(-x^2)$ in the presence of Gaussian noise $\epsilon_j\in\mathcal{N}(0,\sigma^2)$ with $\sigma=0.15$, via $\mathcal{T}_{L}$, $\mathcal{F}_L$, and $\mathcal{L}^{\lambda}_{L}$. We set $N = 300$, $L = 250$, $\lambda=10^{-1}$ and all $\mu_{\ell}$ be 1, and we adopt normalized Legendre polynomials to approximate $f$ in this experiment, which reports excellent denoising ability of $\mathcal{L}^{\lambda}_L$. 
\begin{figure}[htbp]
  \centering
  \includegraphics[width=\textwidth]{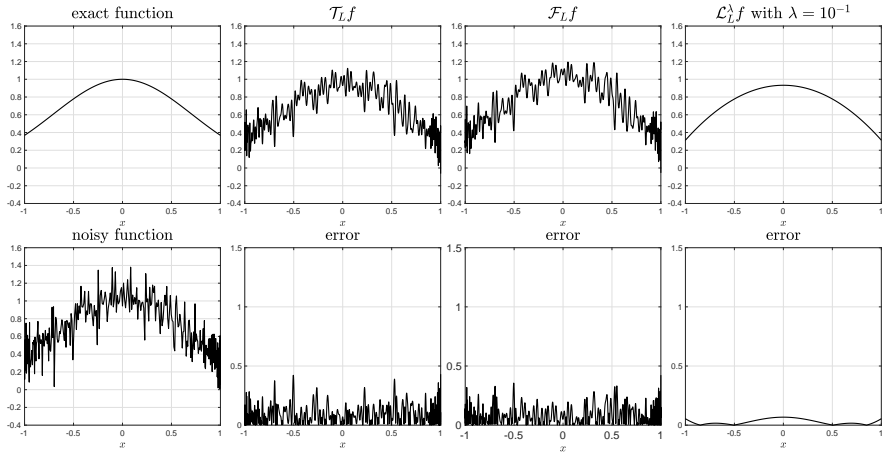}\\
  \caption{Approximation results of $f(x)=\exp(-x^2)$ over $[-1,1]$ via $\mathcal{T}_{L}$, $\mathcal{F}_{L}$, and $\mathcal{L}^{\lambda}_{L}$.}\label{figure1}
\end{figure}

Table \ref{tab:interval} reports the $L_2$ approximation errors of the same function but with respect to different $\sigma$ which describes the level of noise and different $\lambda$. Besides, the sparsity of Lasso hyperinterpolation coefficients is also reported. In this experiment, there should be 251 coefficients in constructing $\mathcal{T}_{250}f$, $\mathcal{F}_{250}f$, and $\mathcal{L}^{\lambda}_{250}f$. For each setting of $\lambda$ and $\sigma$, we test 5 times and report the average values. It is shown in this table that $\mathcal{L}^{\lambda}_{L}$ enjoys the leading position in removing Gaussian noise on the interval. For a fixed level of noise, an appropriate $\lambda=10^{-1}$ leads to $\|\mathcal{L}_L^{\lambda}f^{\epsilon}-f\|_2\approx 0.0731$, which is approximately one third of the $L_2$ errors of $\mathcal{T}_{L}$ and $\mathcal{F}_{L}$, respectively. For a fixed $\lambda$, unlike $\mathcal{T}_{L}$ and $\mathcal{F}_{L}$, the Lasso hyperinterpolation $\mathcal{L}^{\lambda}_{L}$ shows robustness with respect to the increase in the level of noise. For the sparsity of $\bm{\beta}$, it is illustrated in this table that decreasing $\lambda$ and increasing the level of noise both increase the number of nonzero entries of $\bm{\beta}$.
\begin{table}[htbp]
  \centering
  \setlength{\abovecaptionskip}{0pt}
\setlength{\belowcaptionskip}{10pt}
  \caption{Approximation errors and the sparsity of Lasso hyperinterpolation coefficients of $f(x)=\exp(-x^2)$ over $[-1,1]$ via $\mathcal{T}_{L}$, $\mathcal{F}_{L}$, and $\mathcal{L}^{\lambda}_{L}$, with different values of $\lambda$ and different standard derivation $\sigma$ of Gaussian noise added on.}\label{tab:interval}
  \begin{tabular}{c|cccc|cccc}
   \cline{1-9} 
    ~ & \multicolumn{4}{c|}{$\sigma=0.2$, and $\lambda$ takes} & \multicolumn{4}{c}{$\lambda=10^{-1}$, and $\sigma$ takes} \\
        & $10^{-0.8}$ & $10^{-1}$ & $10^{-1.5}$ & $10^{-2}$  & 0.1 & 0.15 & 0.2 & 0.25  \\ \hline
    \small{Tikhonov}   & 0.2645  &  0.2369  &  0.2431  &  0.2450  &  0.1462  &  0.1831  &  0.2440  &  0.2867 \\
    Filtered  & 0.2097  &  0.2161  &  0.2202  &  0.2150  &  0.1057  &  0.1536  &  0.2236  &  0.2663 \\
    Lasso     & 0.1454  &  0.0731  &  0.1114  &  0.2017  &  0.0811  &  0.0733  &  0.0802  & 0.0890 \\
    \hline
    $\|\bm{\beta}\|_0$ & 2 & 2.8 & 89.8 & 192.6 & 2 & 2.2 & 2.8 & 6.2\\
    \hline
  \end{tabular}
\end{table}

\subsection{The disc}
We then consider $\Omega=\{\mathbf{x}\in\mathbb{R}^2:\mathbf{x}=(x_1,x_2)\text{ and }x_1^2+x_2^2\leq1\}$, which is a unit disc on $\mathbb{R}^2$, with $\text{d}\omega=(1/\pi)\text{d}\mathbf{x}$. Thus $$V=\int_{\Omega}\text{d}\omega=\frac{1}{\pi}\int_0^1\int_0^{2\pi}1r\text{d}\theta\text{d}r=1.$$ 
In this case, $\mathbb{P}_L:=\left\{\sum_{j=0}^L\sum_{k=0}^jb_{jk}x_1^kx_2^{j-k}:b_{jk}\in\mathbb{R}\right\}$ is a linear space of polynomials of degree at most $L$ on the unit disc, and hence $d=\binom{L+2}{2}=(L+2)(L+1)/2$.

Fix $L$ as the degree of Lasso hyperinterpolation polynomial, and let $\{\Lambda_{\ell}:1\leq \ell\leq (L+2)(L+1)/2\}$ be a family of ridge polynomials on the unit disc, which were introduced by Logan and Shepp \cite{logan1975optimal}. If we write $\mathbf{x}=(r,\theta)$, where $r$ and $\theta$ are the radius and azimuthal directions of $\mathbf{x}$, respectively, then the discrete inner product \eqref{equ:discreteinnerproduct} can be expressed as 
\begin{equation*}\begin{split}
\left<v,z\right>_{N}:=&\frac{1}{\pi}\sum_{j=0}^N\sum_{m=0}^{2N}v\left(r_j,\frac{2\pi m}{2N+1}\right)z\left(r_j,\frac{2\pi m}{2N+1}\right)w_{j}\frac{2\pi}{2N+1}r_j\\
&=\sum_{j=0}^N\sum_{m=0}^{2N}v\left(r_j,\frac{2\pi m}{2N+1}\right)z\left(r_j,\frac{2\pi m}{2N+1}\right)w_{j}\frac{2}{2N+1}r_j,
\end{split}\end{equation*}
where we use the trapezoidal rule for the azimuthal direction and the Gauss-Legendre quadrature rule over $[0,1]$ for the radial direction, that is, $\{r_j\}_{j=0}^N$ and $\{w_j\}_{j=0}^N$ are Gauss-Legendre quadrature points and weights, respectively. Such an inner product was constructed in \cite{hansen2009norm}, and is exact for all $v,z\in\mathbb{P}_N$. Hence 
\begin{equation*}
\sum_{j=0}^N\sum_{m=0}^{2N}g\left(r_j,\frac{2\pi m}{2N+1}\right)w_{j}\frac{2}{2N+1}r_j=\frac{1}{\pi}\int_0^1\int_0^{2\pi}g(r,\theta)r\text{d}\theta\text{d}r \text{ for all }g\in\mathbb{P}_{2N}.
\end{equation*}

Then the Lasso hyperinterpolation \eqref{equ:lassohyperinterpolation} becomes
\begin{equation*}
\mathcal{L}_L^{\lambda}f:=\sum_{\ell=0}^L\mathcal{S}_{\lambda\mu_{\ell}}\left(\sum_{j=0}^N\sum_{m=0}^{2N}f\left(r_j,\frac{2\pi m}{2N+1}\right)\Lambda_{\ell}\left(r_j,\frac{2\pi m}{2N+1}\right)w_{j}\frac{2}{2N+1}r_j\right)\Lambda_{\ell}.
\end{equation*}

Figure \ref{figure2} displays the approximation results of function $f(x_1,x_2)=(1-(x_1^2+x_2^2))\exp(x_1\cos(x_2))$ contaminated by some single impulse noise, via $\mathcal{T}_{L}$, $\mathcal{F}_{L}$, and $\mathcal{L}^{\lambda}_{L}$. This kind of noise $\epsilon_j$, in our experiments, takes a uniformly distributed random values in $[-a,a]$ with probability $1/2$, which is generated by MATLAB command \texttt{a*(1-2*rand(1))*binornd(1,0.5)}. In Figure \ref{figure2} we let $a = 3.5$. This function is the true solution of a nonlinear Poisson equation as seen in \cite{atkinson2005solving}, which was solved by hyperinterpolation-based spectral methods in \cite{hansen2009norm}. We set $N = 135$ (136 quadrature points), $L = 16$, $\lambda=10^{-1.5}$ and all $\mu_{\ell}$ be 1 in this experiment, which also reports exciting denoising ability of $\mathcal{L}^{\lambda}_L$. Errors near the boundary of the disc are in a good agreement with the theoretical analysis in \cite[Section 4]{hansen2009norm} that for a given $L$, pointwise errors near the boundary are larger than those around the center of the disc.
\begin{figure}[htbp]
  \centering
  \includegraphics[width=\textwidth]{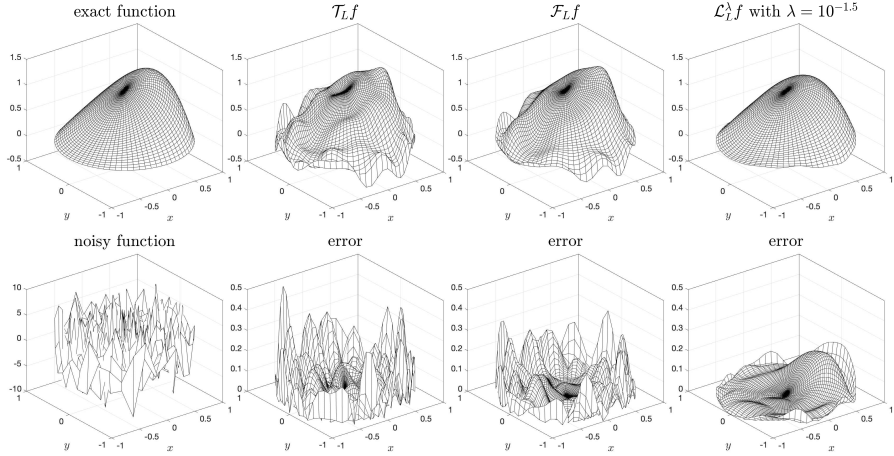}\\
  \caption{Approximation results of $f(x_1,x_2)=(1-(x_1^2+x_2^2))\exp(x_1\cos(x_2))$ over the unit disc via $\mathcal{T}_{L}$, $\mathcal{F}_{L}$, and $\mathcal{L}^{\lambda}_{L}$.}\label{figure2}
\end{figure}

Table \ref{tab:disc} reports the $L_2$ approximation errors of the same function but with respect to different values of $\lambda$ and different level $a$ of noise. Besides, the sparsity of Lasso hyperinterpolation coefficients is also reported. In this experiment, there should be 153 coefficients in constructing $\mathcal{T}_{16}f$, $\mathcal{F}_{16}f$, and $\mathcal{L}_{16}^{\lambda}f$. For each setting of $\lambda$ and $a$, we test 5 times and report the average values. Table \ref{tab:disc} asserts the denoising ability of $\mathcal{L}^{\lambda}_{L}$ with respect to impulse noise and the robustness of $\mathcal{L}^{\lambda}_{L}$ with respect to the increasing level of noise. We also note that increasing the level of impulse noise may not increase the number of nonzero entries of $\bm{\beta}$ as significantly as the case of intervals.
\begin{table}[htbp]
  \centering
  \setlength{\abovecaptionskip}{0pt}
\setlength{\belowcaptionskip}{10pt}
  \caption{Approximation errors and the sparsity of Lasso hyperinterpolation coefficients of $f(x_1,x_2)=(1-(x_1^2+x_2^2))\exp(x_1\cos(x_2))$ over a unit disc via $\mathcal{T}_{L}$, $\mathcal{F}_{L}$, and $\mathcal{L}^{\lambda}_{L}$, with different values of $\lambda$ and different values of $a$ used in generating single impulse noise.}\label{tab:disc}
  \begin{tabular}{c|cccc|cccc}
  \cline{1-9}
     & \multicolumn{4}{c|}{$a=3.5$, and $\lambda$ takes} & \multicolumn{4}{c}{$\lambda=10^{-1.2}$, and $a$ takes} \\
        & $10^{-1.5}$ & $10^{-1.3}$ & $10^{-1.1}$ & $10^{-0.9}$  & 2.5 & 3 & 3.5 & 4  \\ \hline
    \small{Tikhonov}   & 0.4424  &  0.4296  &  0.3899  &  0.3616  &  0.4133  &  0.4200  &  0.4101  &  0.4088 \\
    Filtered  & 0.4840  &  0.4948  &  0.4864  &  0.5080  &  0.4942  &  0.5019  &  0.4887  &  0.4880 \\
    Lasso     & 0.4234  &  0.3947  &  0.3455  &  0.2922  &  0.3745  &  0.3832  &  0.3694  &  0.3669 \\
    \hline
    $\|\bm{\beta}\|_0$ & 10.4 & 8 & 7.6 & 5.4 & 8 & 8 & 8 & 8\\
    \hline
  \end{tabular}
\end{table}

\subsection{The sphere}
We then take $\Omega=\mathbb{S}^2\subset\mathbb{R}^3$ with $\text{d}\omega=\omega(\mathbf{x})\text{d}\mathbf{x}$, where $\omega(\mathbf{x})$ is an area measure on $\mathbb{S}^2$. Since $V=\int_{\mathbb{S}^2}\omega(\mathbf{x})\text{d}\mathbf{x}$ denotes the surface area of $\mathbb{S}^2$, we have $$V=4\pi.$$  Here $\mathbb{P}_L(\Omega):=\mathbb{P}_L(\mathbb{S}^2)$ is the space of spherical polynomials of degree at most $L$. Let the basis be a set of orthonormal spherical harmonics \cite{muller1966spherical} $\{Y_{\ell,k}:\ell=0,1\ldots,L,k=1,\ldots,2\ell+1\}$, and the dimension of $\mathbb{P}_L(\mathbb{S}^2)$ is $d=\dim\mathbb{P}_L=(L+1)^2$. There are many quadrature rules \cite{an2010well,atkinson2019spectral,womersley2001good} satisfying 
\begin{equation}\label{equ:quadraturesphere}
\sum_{j=1}^Nw_jg(\mathbf{x}_j)=\int_{\mathbb{S}^2}g{\rm{d}}\omega\quad\forall g\in\mathbb{P}_{2L}
\end{equation}
for a spherical polynomial $g$, then the Lasso hyperinterpolation \eqref{equ:lassohyperinterpolation} becomes
\begin{equation}\label{equ:lassohyperinterpolationonsphere}
\mathcal{L}_L^{\lambda}f:=\sum_{\ell=0}^L\sum_{k=1}^{2\ell+1}\mathcal{S}_{\lambda\mu_{\ell}}\left(\sum_{j=1}^Nw_jf(\mathbf{x}_j)Y_{\ell,k}(\mathbf{x}_j)\right)Y_{\ell,k}.
\end{equation}
When $\lambda\rightarrow0$ and $f^{\epsilon}=f$, our error bounds of $\mathcal{L}^{\lambda}_L$ reduce into $\|\mathcal{L}_Lf(\mathbf{x})-f(\mathbf{x})\|_2\leq4\pi^{1/2}E_{L}(f)$, which coincides with the bound given by Sloan \cite{sloan1995polynomial}. 

We provide a concrete quadrature for \eqref{equ:quadraturesphere}, \emph{spherical $t$-design}, which was introduced by Delsarte, Goethals, and Seidel \cite{delsarte1991geometriae} in 1977. A point set $\{\mathbf{x}_1,\ldots,\mathbf{x}_N\}\subset\mathbb{S}^2$ is a \emph{spherical $t$-design} if it satisfies
\begin{equation}\label{equ:std}
\frac1N\sum_{j=1}^Np(\mathbf{x}_j)=\frac{1}{4\pi}\int_{\mathbb{S}^2}p(\mathbf{x})\text{d}\omega(\mathbf{x})\quad\forall p\in\mathbb{P}_t.
\end{equation}
In other words, it is a set of points on the sphere such that a equal-weight quadrature rule at these points integrates all (spherical) polynomials up to degree $t$ exactly. 

Figure \ref{figure4} displays the approximation results via $\mathcal{T}_{L}$, $\mathcal{F}_{L}$ and $\mathcal{L}^{\lambda}_{L}$ of a function $f$ defined below, perturbed by mixed Gaussian noise with $\sigma=0.015$ and single impulse noise with $a=0.02$. Let $\mathbf{z}_1=[1,0,0]^{\rm{T}}$, $\mathbf{z}_2=[-1,0,0]^{\rm{T}}$, $\mathbf{z}_3=[0,1,0]^{\rm{T}}$, $\mathbf{z}_4=[0,-1,0]^{\rm{T}}$, $\mathbf{z}_5=[0,0,1]^{\rm{T}}$, and $\mathbf{z}_6=[0,0,-1]^{\rm{T}}$, the testing function $f$ is defined as 
\begin{equation}\label{equ:wendland}
f(\mathbf{x})=\sum_{i=1}^6\phi_2(\|\mathbf{z}_i-\mathbf{x}\|_2),
\end{equation}
where $\phi_2(r):=\tilde{\phi}_2\left(r/\delta_2\right)$ is a normalized Wendland function \cite{chernih2014wendland}, with  $\tilde{\phi}_2(r):=\left(\max\{1-r,0\}\right)^6(35r^2+18r +3)/3$ been an original Wendland function \cite{wendland1995piecewise} and $\delta_2=(9\Gamma(5/2))/(2\Gamma(3))$. In this experiment we employ well conditioned spherical $t$-designs \cite{an2010well}, which are designed to have good geometry properties for integration and regularized least squares approximation \cite{an2012regularized}. Let $\mathcal{X}_{N}$ be a \emph{well conditioned spherical $t$-design} with $N=(t+1)^2$, $L=15$, $t=2L=30$, $\lambda=10^{-2.5}$ and all $\mu_{\ell}$ be 1 \cite{an2010well}. We also use an advanced Tikhonov regularized least squares approximation technique for comparison, which incorporates with Laplace-Beltrami operators in order to reduce noise \cite{an2012regularized}. As Laplace-Beltrami operators are adopted, a relatively small $\lambda=10^{-3.5}$ is desired for $\mathcal{T}_L$. Figure \ref{figure4} shows the denoising ability of Lasso hyperinterpolation in the approximation of $f$, which is considered onerous with respect to noise reduction tasks as of its limits smoothness at the centers $\mathbf{z}_i$ and at the boundary of each cap (see the exact function) with center $\mathbf{z}_i$ \cite{gia2010multiscale,wang2017fully}. 
\begin{figure}[htbp]
  \centering
  \includegraphics[width=\textwidth]{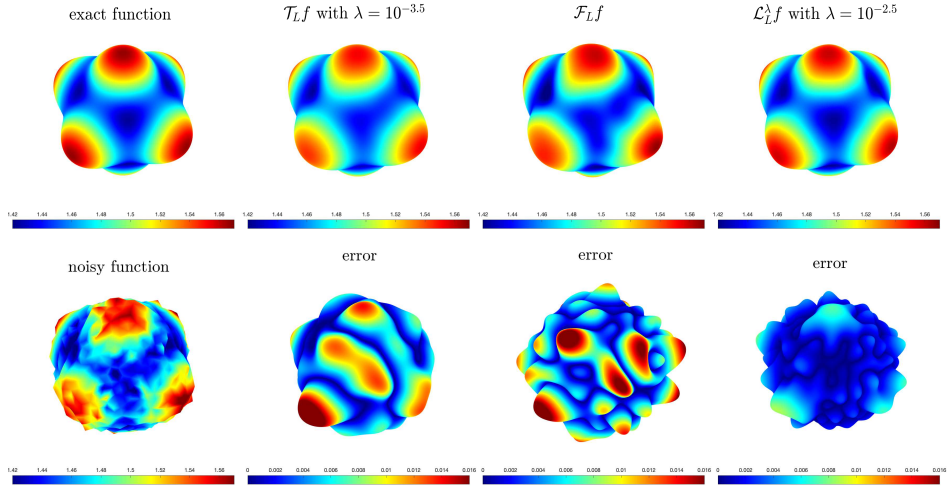}\\
  \caption{Approximation results of $f(\mathbf{x})=\sum_{i=1}^6\phi_2(\|\mathbf{z}_i-\mathbf{x}\|_2)$ over the unit sphere via $\mathcal{T}_{L}$, $\mathcal{F}_{L}$, and $\mathcal{L}^{\lambda}_{L}$.}\label{figure4}
\end{figure}

Table \ref{tab:sphere} reports the $L_2$ approximation errors of the same function but with respect to different values of $\lambda$ and different $\sigma$ which describes the level of Gaussian noise. The level $a=0.02$ of single impulse noise is fixed. Besides, the sparsity of Lasso hyperinterpolation coefficients is also reported. In this experiment, there should be 256 coefficients in constructing $\mathcal{T}_{15}f$, $\mathcal{F}_{15}f$, and $\mathcal{L}_{15}^{\lambda}f$. For each setting of $\lambda$ and $a$, we test 5 times and report the average values. Table \ref{tab:sphere} shows the denoising ability of $\mathcal{L}^{\lambda}_{L}$ with respect to mixed Gaussian noise and impulse noise. Even comparing with the Tikhonov least squares approximation making use of the Laplace-Beltrami operator, which reported satisfying denoising ability in \cite{an2012regularized}, $\mathcal{L}^{\lambda}_L$ still reports an outperforming approximation quality than $\mathcal{T}_L$ with an appropriate choice of $\lambda$.
\begin{table}[htbp]
  \centering
  \setlength{\abovecaptionskip}{0pt}
  \setlength{\belowcaptionskip}{10pt}
  \caption{Approximation errors and the sparsity of Lasso hyperinterpolation coefficients of a Wendland function \eqref{equ:wendland} over a unit sphere via $\mathcal{T}_{L}$, $\mathcal{F}_{L}$, and $\mathcal{L}^{\lambda}_{L}$, with different values of $\lambda$, fixed $a=0.02$ in generating single impulse noise, and different values of $\sigma$ used in generating Gaussian noise. In columns 6-9, $\lambda=10^{-2.5}$ for $\mathcal{L}^{\lambda}_L$ and $\lambda=10^{-3.5}$ for $\mathcal{T}_L$.}\label{tab:sphere}
  \begin{tabular}{c|cccc|cccc}
  \cline{1-9}
     & \multicolumn{4}{c|}{$\sigma=0.02$, and $\lambda$ takes} & \multicolumn{4}{c}{$\sigma$ takes} \\
        & $10^{-3.5}$ & $10^{-3}$ & $10^{-2.5}$ & $10^{-2}$  & 0.015 & 0.02 & 0.025 & 0.03  \\ \hline
    \small{Tikhonov}   & 0.0064  &  0.0114 &   0.0208  &  0.0301  &  0.0055  &  0.0067  &  0.0066 &   0.0075 \\
    Filtered  &  0.0119 &   0.0112  &  0.0110  &  0.0107  &  0.0089 &   0.0111 &   0.0131  &  0.0158 \\
    Lasso     &  0.0109  &  0.0081  &  0.0037  &  0.0053   &   0.0026 &   0.0039 &   0.0051  &  0.0073 \\
    \hline
    $\|\bm{\beta}\|_0$ & 233.6 & 184.8 & 55 & 3 & 33.6 &   56.8&   74.4&   95.8\\
    \hline
  \end{tabular}
\end{table}

\subsection{The cube}
We consider a unit cube $\Omega=[-1,1]^3\subset\mathbb{R}^3$ with $\text{d}\omega=\omega(\mathbf{x})\text{d}\mathbf{x}$, $\mathbf{x}=[x_1,x_2,x_3]^{\rm{T}}$, where the measure is given by the product Chebyshev weight function $\textrm{d}\omega=\omega(\mathbf{x})\textrm{d}\mathbf{x}$, and $\omega(\mathbf{x}):=(1/\pi^3)\prod_{i=1}^3(1/\sqrt{1-x_i^2})$. Thus in the unit cube, $$V=\int_{[-1,1]^3}\frac{1}{\pi^3}\prod_{i=1}^3\frac{1}{\sqrt{1-x_i^2}}\text{d}\mathbf{x}=1.$$ It is simple to implement quadrature in the cube, actually we should call it a cubature, by tensor products from its one dimensional version. However, we mention a tailored cubature for hyperinterpolation in 3-dimensional cube \cite{caliari2008hyperinterpolation,de2009new}. 
As $\mathbb{P}_L$ is a linear space of polynomials of degree at most $L$ in the cube $[-1,1]^3$, we have $d=\binom{L+3}{3}=(L+3)(L+2)(L+1)/6$. 
Note that for interval $[-1,1]$ and square $[-1,1]^2$, there exist minimal quadrature rules (cf. Definition \ref{def:minimal}) \cite{xu1996lagrange}, but for cube $[-1,1]^3$, the required number of nodes is much greater than the lower bound $d$. Thus the reason why we desire the new cubature in \cite{de2009new} is that the required number $N$ of points for its exactness is only about $2\left(\left\lfloor\frac{L}{2}\right\rfloor\right)^3(1+o(L^{-1}))$, roughly speaking, $N\approx L^3/4$, which is substantially less than its previous cubature rules, for example, see \cite{bojanov1997minimal}.

Let $\{p_{\ell}\}$ be a family of product orthonormal Chebyshev basis \cite{dunkl2014orthogonal} with $p_{\ell}(\mathbf{x}):=\tilde{T}_{\ell_1}(x_1)\tilde{T}_{\ell_2}(x_2)\tilde{T}_{\ell_3}(x_3)$, where $\tilde{T}_k(\cdot)=\sqrt{2}\cos(k\arccos(\cdot))$ for $k>0$ and $\tilde{T}_0(\cdot)=1$, and let $C_L=\{\cos(k\pi/L),k=0,\ldots,L\}$ be the set of $L+1$ Chebyshev-Lobatto points. Then choose a nodes set $\mathcal{X}_L=\left(C_{L+1}^{\rm{E}}\times C_{L+1}^{\rm{E}}\times C_{L+1}^{\rm{E}}\right)\cup\left(C_{L+1}^{\rm{O}}\times C_{L+1}^{\rm{O}}\times C_{L+1}^{\rm{O}}\right)$, 
where $C_{L+1}^{\rm{E}}$ and $C_{L+1}^{\rm{O}}$ are the restriction of $C_{L+1}$ to even (``E'') and odd (``O'') indices, respectively, and the corresponding weights are given by 
\begin{equation*}
w_{\bm{\xi}}:=\frac{4}{(L+1)^3}\begin{cases}
1 &\text{ if }\bm{\xi} \text{ is an interior point},\\
1/2 &\text{ if }\bm{\xi} \text{ is a face point},\\
1/4 &\text{ if }\bm{\xi} \text{ is an edge point},\\
1/8 &\text{ if }\bm{\xi} \text{ is a vertex point}.\\
\end{cases}
\end{equation*}

Fix $L$ as the degree of Lasso hyperinterpolation polynomial by $\ell_1+\ell_2+\ell_3\leq L$, and require the number of nodes, which is about $L^3/4$, to guarantee the exactness of the cubature rule \cite{de2009new}. Let 
\begin{equation}\label{equ:functionf}
F(\bm{\xi})=F(\xi_1,\xi_2,\xi_3)=\begin{cases}
w_{\bm{\xi}}f(\bm{\xi}),& \bm{\xi}\in\mathcal{X}_L,\\
0,&\bm{\xi}\in\left(C_{L+1}\times C_{L+1}\times C_{L+1}\right)\backslash\mathcal{X}_L.
\end{cases}
\end{equation}
The Lasso hyperinterpolation \eqref{equ:lassohyperinterpolation} becomes
\begin{equation}\label{equ:lassohyperinterpolationoncube}
\mathcal{L}_L^{\lambda}f:=\sum\limits_{\ell_1+\ell_2+\ell_3\leq L}\mathcal{S}_{\lambda\mu_{\ell}}\left(\alpha_{\ell}\right)p_{\ell}
\end{equation}
with hyperinterpolation coefficients 
\begin{equation*}
\alpha_{\ell}=\gamma_{\ell}\sum_{i=0}^{L+1}\left(   \sum_{j=0}^{L+1}\left(   \sum_{k=0}^{L+1}F_{ijk}\cos\frac{k\ell_1\pi}{L+1}     \right)\cos\frac{j\ell_2\pi}{L+1}         \right)\cos\frac{i\ell_3\pi}{L+1},
\end{equation*}
where $F_{ijk}=F\left(\cos\frac{i\pi}{L+1},\cos\frac{j\pi}{L+1},\cos\frac{k\pi}{L+1}\right)$, $0\leq i,j,k\leq L+1$, and
\begin{equation*}
\gamma_{\ell}=\prod_{s=1}^3\gamma_{\ell_s},\quad\gamma_{\ell_s}=\begin{cases}
\sqrt{2},&\ell_s>0,\\
1,&\ell_s=0,
\end{cases}
\quad s = 1,2,3.
\end{equation*}

\begin{remark}
The above derivation is based on a special case of the new cubature in \cite{de2009new}. Actually there are another three cases in a cube, but authors of \cite{de2009new} stated that numerical behaviors of these cubature rules should be the same. Hence \eqref{equ:lassohyperinterpolationoncube} is also special version of Lasso hyperinterpolation in $[-1,1]^3$ of degree $L$ based on the cubature in \cite{de2009new}, which finally computes $(L+1)(L+2)(L+3)/6\approx L^3/6$ coefficients and requires about $L^3/4$ nodes. 
\end{remark}

We test function $f(x,y,z)=\exp(-1/(x^2+y^2+z^2))$ contaminated by some noise (all nonzero values of the function are perturbed by Gaussian noise of standard derivation $\sigma=0.2$), via $\mathcal{T}_{L}$, $\mathcal{F}_{L}$, and $\mathcal{L}^{\lambda}_{L}$. It is hard to display the approximation results in a cube, which are in a format of 3D volumetric data. Figure \ref{figure3} displays these data along some slice planes, including $x=-0.25$, $x=0.5$, $x=1$, $y=0$, $y=1$, $z=-1$, and $z=0$, which provides a window into approximation results of the test function. Set $L = 50$, $\lambda=10^{-2.5}$ and all $\mu_{\ell}$ be 1 in this experiment, in which 33150 quadrature points are required to ensure the exactness \eqref{equ:exactquadrature} of the cubature rule. Figure \ref{figure3} illustrates great recovering ability of contaminated $f$ in the presence of noise. 
\begin{figure}[htbp]
  \centering
  \includegraphics[width=\textwidth]{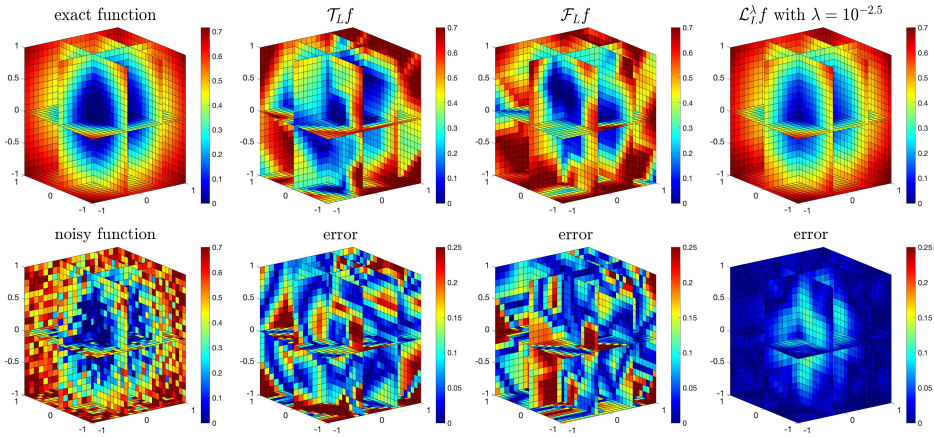}\\
  \caption{Approximation results of $f(x,y,z)=\exp(-1/(x^2+y^2+z^2))$ in the unit cube via $\mathcal{T}_{L}$, $\mathcal{F}_{L}$, and $\mathcal{L}^{\lambda}_{L}$.}\label{figure3}
\end{figure}

Table \ref{tab:cube} reports the $L_2$ approximation errors of the same function but with respect to different values of $\lambda$ and different $\sigma$ which describes the level of noise. Besides, the sparsity of Lasso hyperinterpolation coefficients is also reported. In this experiment, there should be 22100 coefficients in constructing $\mathcal{T}_{50}f$, $\mathcal{F}_{50}f$, and $\mathcal{L}_{50}^{\lambda}f$. For each setting of $\lambda$ and $\sigma$, we test 5 times and report the average values. Comparing with $\mathcal{T}_{L}$ and $\mathcal{F}_{L}$, Table \ref{tab:cube} shows much better denoising quality of $\mathcal{L}^{\lambda}_{L}$ with respect to Gaussian noise and more robustness of $\mathcal{L}^{\lambda}_{L}$ with respect to an increasing level of noise. The sparsity of $\bm{\beta}$ increases as the level of noise increasing, but this process seems to be more sensitive than that on the other three manifolds, which is due to a larger number (22100) of coefficients in constructing $\mathcal{T}_{50}f$.
\begin{table}[htbp]
  \centering
  \setlength{\abovecaptionskip}{0pt}
\setlength{\belowcaptionskip}{10pt}
  \caption{Approximation errors and the sparsity of Lasso hyperinterpolation coefficients of $f(x,y,z)=\exp(-1/(x^2+y^2+z^2))$ in a unit cube via $\mathcal{T}_{L}$, $\mathcal{F}_{L}$, and $\mathcal{L}^{\lambda}_{L}$, with different values of $\lambda$ and different values of $\sigma$ used in generating Gaussian noise.}\label{tab:cube}
  \begin{tabular}{c|cccc|cccc}
  \cline{1-9}
     & \multicolumn{4}{c|}{$\sigma=0.2$, and $\lambda$ takes} & \multicolumn{4}{c}{$\lambda=10^{-2.5}$, and $\sigma$ takes} \\
        & $10^{-2.5}$ & $10^{-2.4}$ & $10^{-2.3}$ & $10^{-2}$  & 0.002 & 0.05 & 0.2 & 0.4  \\ \hline
    \small{Tikhonov}   & 11.056 &  11.828 &  11.951 &  12.198 &  0.1469  &  3.0025  & 11.992 &  23.279 \\
    Filtered  & 11.557 &  10.758 &  11.436 &  11.337  &  0.1131 &   2.8722 &  11.241  & 24.725 \\
    Lasso     &  0.9744 &   0.9761 &   1.1014 &   1.7302  &  0.8764  &  0.8776 &   0.9870  &  6.2842 \\
    \hline
    $\|\bm{\beta}\|_0$ & 104.4 & 19 & 10.4 & 7 & 11 & 11 & 104.6 & 3353.4\\
    \hline
  \end{tabular}
\end{table}

\section{Final remarks}
In this paper, we introduce a novel approximation scheme $\mathcal{L}^{\lambda}_L$ for function approximation with noisy data, and derive general theory (error estimation) for it on general regions. The theory applies to four particular manifolds well, including an interval, a disc, a 2-sphere, and a 3-cube, but it is also shown that there exist obvious differences from manifold to manifold. Let us list some of them. In theory, from discussions on the case of smooth functions in Section \ref{sec:discussion}, additional geometric properties of $\Omega$ could relax the assumption on $f$ from $f\in\mathcal{C}^{k,\zeta}(\Omega)$ to $f\in\mathcal{C}^{k}(\Omega)$ with the convergence order $\mathcal{O}(L^{-k})$ maintained. In a numerical perspective, approximation qualities are sensitive to the level of noise in a 3-cube, but they are not so sensitive on the other three manifolds. On a 2-sphere, as spherical harmonics have an intrinsic characterization as the eigenfunctions of the Laplace-Beltrami operator, the operator can be incorporated into Tikhonov regularized least squares approximation \cite{an2012regularized}, which provides a constructive approach for function approximation with noisy data. Hence when $\mathcal{L}^{\lambda}_L$ is used in real-world applications, such as geomathematics (the earth’s potato shape can be mapped to a 2-sphere by an appropriate smooth mapping) and hydrocarbon exploration (modeled in a 3-cube), one shall take the geometric properties of particular $\Omega$ into account.

Our approach is achieved by a soft threshold operator, and our theory is derived with assumptions that additive noise $\epsilon_j$ is sub-Gaussian and $f\in\mathcal{C}(\Omega)$ or $\mathcal{C}^{k,\zeta}(\Omega)$ with $0<\zeta\leq1$. We may survey some of our results which can be improved or extended further. The $\ell_1$ regularization (Lasso) corresponds to a soft threshold operator, which is studied in this paper, but one may consider other types of threshold operators, e.g. a hard threshold operator \cite{foucart2011hard}. We adopt sub-Gaussian random variables to model noise, but from our discussion on noise in Section \ref{sec:noise}, once we know the distribution of $\epsilon_j$, we can estimate the expectation of $\|\bm{\epsilon}\|_{\infty}$ analytically. Thus other types of noise may also be studied, for example, noise modeled by sub-exponential random variables, including Rayleigh noise, gamma noise, and exponential noise. We may also consider other function spaces which measure the smoothness of $f$. An important direction is to consider some Sobolev spaces \cite{dai2011polynomial,MR2274179}. In this case, error estimations of $\mathcal{L}^{\lambda}_L$ may be derived and controlled by Sobolev norms of $f$ rather than uniform norms. We only investigate four low-dimensional manifolds in this paper; however, some other low- and high-dimensional manifolds could be considered, e.g. $s$-cubes, $(s-1)$-spheres, $s$-balls, and so forth. We refer to \cite{wang2014norm} for $s$-cubes, \cite{le2001uniform,sloan2011polynomial,sloan2012filtered} for $(s-1)$-spheres, and \cite{wade2013hyperinterpolation} for $s$-balls.

\section*{Acknowledgments}
We thank Professor Kendall E. Atkinson of the University of Iowa for providing us with MATLAB codes of disc-related experiments, which were conducted in \cite{hansen2009norm}. We are very grateful to the anonymous referees for their careful reading of our manuscript and their many insightful comments.

\bibliographystyle{siamplain}
\bibliography{references}

\begin{thebibliography}{10}

\bibitem{an2010well}
{\sc C.~An, X.~Chen, I.~H. Sloan, and R.~S. Womersley}, {\em Well conditioned
  spherical designs for integration and interpolation on the two-sphere}, SIAM
  J. Numer. Anal., 48 (2010), pp.~2135--2157,
  \url{https://doi.org/10.1137/100795140}.

\bibitem{an2012regularized}
{\sc C.~An, X.~Chen, I.~H. Sloan, and R.~S. Womersley}, {\em Regularized least
  squares approximations on the sphere using spherical designs}, SIAM J. Numer.
  Anal., 50 (2012), pp.~1513--1534, \url{https://doi.org/10.1137/110838601}.

\bibitem{an2020tikhonov}
{\sc C.~An and H.-N. Wu}, {\em Tikhonov regularization for polynomial
  approximation problems in {G}auss quadrature points}, Inverse Problems, 37
  (2021), 015008 (19~pages), \url{https://doi.org/10.1088/1361-6420/abcd44}.

\bibitem{atkinson2019spectral}
{\sc K.~Atkinson, D.~Chien, and O.~Hansen}, {\em Spectral Methods Using
  Multivariate Polynomials on The Unit Ball}, CRC Press, Boca Raton, 2019.

\bibitem{MR2511061}
{\sc K.~Atkinson and W.~Han}, {\em Theoretical Numerical Analysis. A Functional
  Analysis Framework}, Springer, Dordrecht, third~ed., 2009.

\bibitem{atkinson2005solving}
{\sc K.~Atkinson and O.~Hansen}, {\em Solving the nonlinear {P}oisson equation
  on the unit disk}, J. Integral Equations Appl., 17 (2005), pp.~223--241,
  \url{https://doi.org/10.1216/jiea/1181075333}.

\bibitem{bojanov1997minimal}
{\sc B.~Bojanov and G.~Petrova}, {\em On minimal cubature formulae for product
  weight functions}, J. Comput. Appl. Math., 85 (1997), pp.~113--121,
  \url{https://doi.org/10.1016/S0377-0427(97)00133-7}.

\bibitem{sobolev}
{\sc H.~Brezis}, {\em Functional Analysis, {S}obolev Spaces and Partial
  Differential Equations}, Springer, New York, 2011.

\bibitem{caliari2007hyperinterpolation}
{\sc M.~Caliari, S.~De~Marchi, and M.~Vianello}, {\em Hyperinterpolation on the
  square}, J. Comput. Appl. Math., 210 (2007), pp.~78--83,
  \url{https://doi.org/10.1016/j.cam.2006.10.058}.

\bibitem{caliari2008hyperinterpolation}
{\sc M.~Caliari, S.~De~Marchi, and M.~Vianello}, {\em Hyperinterpolation in the
  cube}, Comput. Math. Appl., 55 (2008), pp.~2490--2497,
  \url{https://doi.org/10.1016/j.camwa.2007.10.003}.

\bibitem{chernih2014wendland}
{\sc A.~Chernih, I.~H. Sloan, and R.~S. Womersley}, {\em Wendland functions
  with increasing smoothness converge to a {G}aussian}, Adv. Comput. Math., 40
  (2014), pp.~185--200, \url{https://doi.org/10.1007/s10444-013-9304-5}.

\bibitem{dai2006generalized}
{\sc F.~Dai}, {\em On generalized hyperinterpolation on the sphere}, Proc.
  Amer. Math. Soc., 134 (2006), pp.~2931--2941,
  \url{https://doi.org/10.1090/S0002-9939-06-08421-8}.

\bibitem{dai2011polynomial}
{\sc F.~Dai and Y.~Xu}, {\em Polynomial approximation in {S}obolev spaces on
  the unit sphere and the unit ball}, J. Approx. Theory, 163 (2011),
  pp.~1400--1418, \url{https://doi.org/10.1016/j.jat.2011.05.001}.

\bibitem{de2009new}
{\sc S.~De~Marchi, M.~Vianello, and Y.~Xu}, {\em New cubature formulae and
  hyperinterpolation in three variables}, BIT, 49 (2009), pp.~55--73,
  \url{https://doi.org/10.1007/s10543-009-0210-7}.

\bibitem{delsarte1991geometriae}
{\sc P.~Delsarte, J.-M. Goethals, and J.~J. Seidel}, {\em Spherical codes and
  designs}, Geometriae Dedicata, 6 (1977), pp.~363--388,
  \url{https://doi.org/10.1007/bf03187604}.

\bibitem{dunkl2014orthogonal}
{\sc C.~F. Dunkl and Y.~Xu}, {\em Orthogonal Polynomials of Several Variables},
  Cambridge University Press, Cambridge, 2014.

\bibitem{erdHos1937interpolation}
{\sc P.~Erd{\H{o}}s and P.~Tur{\'a}n}, {\em On interpolation {I}. quadrature
  and mean convergence in the {Lagrange} interpolation}, Ann. of Math., 38
  (1937), pp.~142--155, \url{https://doi.org/10.2307/1968516}.

\bibitem{filbir2008polynomial}
{\sc F.~Filbir and W.~Themistoclakis}, {\em Polynomial approximation on the
  sphere using scattered data}, Math. Nachr., 281 (2008), pp.~650--668,
  \url{https://doi.org/10.1002/mana.200710633}.

\bibitem{foucart2011hard}
{\sc S.~Foucart}, {\em Hard thresholding pursuit: an algorithm for compressive
  sensing}, SIAM J. Numer. Anal., 49 (2011), pp.~2543--2563,
  \url{https://doi.org/10.1137/100806278}.

\bibitem{gautschi2004orthogonal}
{\sc W.~Gautschi}, {\em Orthogonal Polynomials: Computation and Approximation},
  Oxford University Press, Oxford, 2004.

\bibitem{graham2002fully}
{\sc I.~G. Graham and I.~H. Sloan}, {\em Fully discrete spectral boundary
  integral methods for {H}elmholtz problems on smooth closed surfaces in
  $\mathbb{R}^3$}, Numer. Math., 92 (2002), pp.~289--323,
  \url{https://doi.org/10.1007/s002110100343}.

\bibitem{hansen2009norm}
{\sc O.~Hansen, K.~Atkinson, and D.~Chien}, {\em On the norm of the
  hyperinterpolation operator on the unit disc and its use for the solution of
  the nonlinear {P}oisson equation}, IMA J. Numer. Anal., 29 (2009),
  pp.~257--283, \url{https://doi.org/10.1093/imanum/drm052}.

\bibitem{MR2274179}
{\sc K.~Hesse and I.~H. Sloan}, {\em Hyperinterpolation on the sphere}, in
  Frontiers in interpolation and approximation, vol.~282 of Pure Appl. Math.,
  Chapman \& Hall/CRC, Boca Raton,, 2007, pp.~213--248.

\bibitem{hesse2017radial}
{\sc K.~Hesse, I.~H. Sloan, and R.~S. Womersley}, {\em Radial basis function
  approximation of noisy scattered data on the sphere}, Numer. Math., 137
  (2017), pp.~579--605, \url{https://doi.org/10.1007/s00211-017-0886-6}.

\bibitem{MR4118851}
{\sc K.~Hesse, I.~H. Sloan, and R.~S. Womersley}, {\em Local {RBF}-based
  penalized least-squares approximation on the sphere with noisy scattered
  data}, J. Comput. Appl. Math., 382 (2021), 113061 (21~pages),
  \url{https://doi.org/10.1016/j.cam.2020.113061}.

\bibitem{le2008localized}
{\sc Q.~T. Le~Gia and H.~N. Mhaskar}, {\em Localized linear polynomial
  operators and quadrature formulas on the sphere}, SIAM J. Numer. Anal., 47
  (2008), pp.~440--466, \url{https://doi.org/10.1137/060678555}.

\bibitem{le2001uniform}
{\sc Q.~T. Le~Gia and I.~H. Sloan}, {\em The uniform norm of hyperinterpolation
  on the unit sphere in an arbitrary number of dimensions}, Constr. Approx., 17
  (2001), pp.~249--265, \url{https://doi.org/10.1007/s003650010025}.

\bibitem{gia2010multiscale}
{\sc Q.~T. Le~Gia, I.~H. Sloan, and H.~Wendland}, {\em Multiscale analysis in
  {S}obolev spaces on the sphere}, SIAM J. Numer. Anal., 48 (2010),
  pp.~2065--2090, \url{https://doi.org/10.1137/090774550}.

\bibitem{lin2019distributed}
{\sc S.-B. Lin, Y.~G. Wang, and D.-X. Zhou}, {\em Distributed filtered
  hyperinterpolation for noisy data on the sphere}, SIAM J. Numer. Anal., 59
  (2021), pp.~634--659, \url{https://doi.org/10.1137/19M1281095}.

\bibitem{logan1975optimal}
{\sc B.~F. Logan and L.~A. Shepp}, {\em Optimal reconstruction of a function
  from its projections}, Duke Math. J., 42 (1975), pp.~645--659,
  \url{https://doi.org/10.1215/S0012-7094-75-04256-8}.

\bibitem{MR0213785}
{\sc G.~G. Lorentz}, {\em Approximation of Functions}, Holt, Rinehart and
  Winston, New York, 1966.

\bibitem{muller1966spherical}
{\sc C.~M{\"u}ller}, {\em Spherical Harmonics}, Springer, Berlin, 1966.

\bibitem{pereverzev2015parameter}
{\sc S.~V. Pereverzev, I.~H. Sloan, and P.~Tkachenko}, {\em Parameter choice
  strategies for least-squares approximation of noisy smooth functions on the
  sphere}, SIAM J. Numer. Anal., 53 (2015), pp.~820--835,
  \url{https://doi.org/10.1137/140964990}.

\bibitem{pieper2009vector}
{\sc M.~Pieper}, {\em Vector hyperinterpolation on the sphere}, J. Approx.
  Theory, 156 (2009), pp.~173--186,
  \url{https://doi.org/10.1016/j.jat.2008.05.002}.

\bibitem{ragozin1970polynomial}
{\sc D.~L. Ragozin}, {\em Polynomial approximation on compact manifolds and
  homogeneous spaces}, Trans. Amer. Math. Soc., 150 (1970), pp.~41--53,
  \url{https://doi.org/10.2307/1995480}.

\bibitem{ragozin1971constructive}
{\sc D.~L. Ragozin}, {\em Constructive polynomial approximation on spheres and
  projective spaces}, Trans. Amer. Math. Soc., 162 (1971), pp.~157--170,
  \url{https://doi.org/10.2307/1995746}.

\bibitem{reimer2000hyperinterpolation}
{\sc M.~Reimer}, {\em Hyperinterpolation on the sphere at the minimal
  projection order}, J. Approx. Theory, 104 (2000), pp.~272--286,
  \url{https://doi.org/10.1006/jath.2000.3454}.

\bibitem{sloan1995polynomial}
{\sc I.~H. Sloan}, {\em Polynomial interpolation and hyperinterpolation over
  general regions}, J. Approx. Theory, 83 (1995), pp.~238--254,
  \url{https://doi.org/10.1006/jath.1995.1119}.

\bibitem{sloan2011polynomial}
{\sc I.~H. Sloan}, {\em Polynomial approximation on spheres-generalizing de la
  {V}all{\'e}e--{P}oussin}, Comput. Methods Appl. Math., 11 (2011),
  pp.~540--552, \url{https://doi.org/10.2478/cmam-2011-0029}.

\bibitem{sloan2000constructive}
{\sc I.~H. Sloan and R.~S. Womersley}, {\em Constructive polynomial
  approximation on the sphere}, J. Approx. Theory, 103 (2000), pp.~91--118,
  \url{https://doi.org/10.1006/jath.1999.3426}.

\bibitem{sloan2012filtered}
{\sc I.~H. Sloan and R.~S. Womersley}, {\em Filtered hyperinterpolation: a
  constructive polynomial approximation on the sphere}, GEM Int. J. Geomath., 3
  (2012), pp.~95--117, \url{https://doi.org/10.1007/s13137-011-0029-7}.

\bibitem{tibshirani1996regression}
{\sc R.~Tibshirani}, {\em Regression shrinkage and selection via the lasso}, J.
  Roy. Statist. Soc. Ser. B, 58 (1996), pp.~267--288,
  \url{https://doi.org/10.1111/j.2517-6161.1996.tb02080.x}.

\bibitem{tibshirani2011regression}
{\sc R.~Tibshirani}, {\em Regression shrinkage and selection via the lasso: a
  retrospective}, J. Roy. Statist. Soc. Ser. B, 73 (2011), pp.~273--282,
  \url{https://doi.org/10.1111/j.1467-9868.2011.00771.x}.

\bibitem{trefethen2013approximation}
{\sc L.~N. Trefethen}, {\em Approximation Theory and Approximation Practice},
  SIAM, Philadelphia, 2013.

\bibitem{MR3837109}
{\sc R.~Vershynin}, {\em High-Dimensional Probability}, Cambridge University
  Press, Cambridge, 2018.

\bibitem{wade2013hyperinterpolation}
{\sc J.~Wade}, {\em On hyperinterpolation on the unit ball}, J. Math. Anal.
  Appl., 401 (2013), pp.~140--145,
  \url{https://doi.org/10.1016/j.jmaa.2012.11.052}.

\bibitem{wang2017filtered}
{\sc H.~Wang and I.~H. Sloan}, {\em On filtered polynomial approximation on the
  sphere}, J. Fourier Anal. Appl., 23 (2017), pp.~863--876,
  \url{https://doi.org/10.1007/s00041-016-9493-7}.

\bibitem{wang2014norm}
{\sc H.~Wang, K.~Wang, and X.~Wang}, {\em On the norm of the hyperinterpolation
  operator on the $d$-dimensional cube}, Comput. Math. Appl., 68 (2014),
  pp.~632--638, \url{https://doi.org/10.1016/j.camwa.2014.07.009}.

\bibitem{wang2017fully}
{\sc Y.~G. Wang, Q.~T. Le~Gia, I.~H. Sloan, and R.~S. Womersley}, {\em Fully
  discrete needlet approximation on the sphere}, Appl. Comput. Harmon. Anal.,
  43 (2017), pp.~292--316, \url{https://doi.org/10.1016/j.acha.2016.01.003}.

\bibitem{wendland1995piecewise}
{\sc H.~Wendland}, {\em Piecewise polynomial, positive definite and compactly
  supported radial functions of minimal degree}, Adv. Comput. Math., 4 (1995),
  pp.~389--396, \url{https://doi.org/10.1007/BF02123482}.

\bibitem{womersley2001good}
{\sc R.~S. Womersley and I.~H. Sloan}, {\em How good can polynomial
  interpolation on the sphere be?}, Adv. Comput. Math., 14 (2001),
  pp.~195--226, \url{https://doi.org/10.1023/A:1016630227163}.

\bibitem{xiang2012error}
{\sc S.~Xiang}, {\em On error bounds for orthogonal polynomial expansions and
  {G}auss-type quadrature}, SIAM J. Numer. Anal., 50 (2012), pp.~1240--1263,
  \url{https://doi.org/10.1137/110820841}.

\bibitem{xu1996lagrange}
{\sc Y.~Xu}, {\em Lagrange interpolation on {C}hebyshev points of two
  variables}, J. Approx. Theory, 87 (1996), pp.~220--238,
  \url{https://doi.org/10.1006/jath.1996.0102}.

\end{thebibliography}
\end{document}